\documentclass[11pt,reqno]{amsart}
\usepackage {amssymb}
\usepackage{verbatim}
\input xy
\xyoption{all}
\input epsf

\newcommand{\sgn}{\mathrm{sgn}}
\newcommand{\triv}{\mathsf{1}}
\newcommand{\de}{\mathrm{det}}
\newcommand{\HC}{\caH\caC}

\newcommand{\co}{{\mathbf{c}}}

\newcommand{\bfv}{\mathbf{v}}

\newcommand{\GL}{\mathrm{GL}}

\newcommand{\Hom}{\mathrm{Hom}}

\newcommand{\ol}{\overline}

\newcommand{\short}{{\mathrm{short}}}
\newcommand{\longg}{{\mathrm{long}}}
\newcommand{\pr}{{\mathrm{pr}}}

\renewcommand{\1}{\triv}

\newcommand{\field}{\mathbb}
\newcommand{\liealgebra}{\mathfrak}

\newcommand{\C}{{\field C}}
\newcommand{\R}{{\field R}}
\newcommand{\B}{{B}}
\newcommand{\Z}{{\field Z}}

\newcommand{\g}{\liealgebra g}

\newcommand{\lam}{\lambda}

\newcommand{\vep}{\varepsilon}

\newcommand{\al}{\alpha}

\renewcommand{\H}{\mathrm{H}}

\newcommand{\lra}{\longrightarrow}
\newcommand{\ra}{\rightarrow}

\newcommand{\wt}{\widetilde}

\newcommand{\ep}{\epsilon}

\newcommand{\eps}{\epsilon}

\newtheorem{prop}{Proposition}[subsection]
\newtheorem{cor}[prop]{Corollary}
\newtheorem{lemma}[prop]{Lemma}
\newtheorem{theorem}[prop]{Theorem}

\newtheorem{corollary}[prop]{Corollary}
\newtheorem{proposition}[prop]{Proposition}

\newtheorem{convention}[prop]{Convention}

\theoremstyle{definition}
\newtheorem{defn}[prop]{Definition}
\newtheorem{remark}[prop]{Remark}
\newtheorem{example}[prop]{Example}

\newtheorem{definition}[prop]{Definition}

\newcommand{\fra}{\mathfrak{a}}

\newcommand{\frg}{\mathfrak{g}}
\newcommand{\frh}{\mathfrak{h}}

\newcommand{\frk}{\mathfrak{k}}
\newcommand{\frl}{\mathfrak{l}}
\renewcommand{\frm}{\mathfrak{m}}
\newcommand{\frn}{\mathfrak{n}}
\newcommand{\fro}{\mathfrak{o}}
\newcommand{\frp}{\mathfrak{p}}

\newcommand{\frs}{\mathfrak{s}}
\newcommand{\frt}{\mathfrak{t}}

\newcommand{\frz}{\mathfrak{z}}

\newcommand{\bbC}{\mathbb{C}}

\newcommand{\bbH}{\mathbb{H}}

\newcommand{\caA}{\mathcal{A}}

\newcommand{\caC}{\mathcal{C}}

\newcommand{\caH}{\mathcal{H}}

\newcommand{\caJ}{\mathcal{J}}

\newcommand{\caO}{\mathcal{O}}

\newcommand{\caX}{\mathcal{X}}

\def\Int{{\rm Int}\, }

\def\ad {\mathop{\hbox {ad}}\nolimits}

\def \det {\mathrm{det}}
\def \End {\mathop{\hbox {End}}\nolimits}

\def\Ind {\mathop{\hbox{Ind}}\nolimits}
\def \ker{\mathop{\hbox{Ker}}\nolimits}

\def\Id {\mathop{\hbox{Id}}\nolimits}

\numberwithin{equation}{section}

\begin{document}

\title[Functors for unitary representations]{Functors for unitary representations of classical real groups and affine Hecke algebras}

\author{DAN CIUBOTARU and PETER E.~TRAPA}
\date{\today}
\address{Department of Mathematics, University of Utah, Salt Lake City, UT 84112-0090}
\email{ciubo@math.utah.edu}
\email{ptrapa@math.utah.edu}

\thanks{DC is partially supported by NSF Grant DMS-0554278.  PT is partially supported by DMS-0554278  and DMS-0554118.
The authors thank Gordan Savin for several helpful conversations, particularly
in connection with Section \ref{s:vs}.}

\begin{abstract}
We define exact functors from categories of Harish-Chandra
modules for certain real classical groups to finite-dimensional 
modules over an associated graded affine Hecke algebra
with parameters.  
 We then study some of the basic
properties of these functors.  In particular, we show that 
they map irreducible spherical representations to irreducible
spherical representations and, moreover, that they preserve
unitarity.  In the case of split classical groups,
we thus obtain a functorial inclusion of the real spherical
unitary dual (with ``real infinitesimal character'')
into the corresponding $p$-adic spherical unitary dual.
\end{abstract}

\maketitle

\section{Introduction}\label{s:1}

In this paper, we define exact functors from categories of
Harish-Chandra modules for certain real classical groups to modules
over an associated graded affine Hecke algebra with parameters.  Our
main results, in increasing order of detail, are that the functors:
(1) map spherical principal series of the real group to spherical
principal series of the Hecke algebra; (2) map irreducible spherical
representations to irreducible spherical representations; and (3) map
irreducible Hermitian (resp.~unitary) spherical representations to
irreducible Hermitian (resp.~unitary) spherical representations.  In
particular, (3) gives a functorial inclusion of the spherical unitary
dual of the real group into the spherical unitary dual of an associated
graded Hecke algebra.  In the split cases, Lusztig's work
\cite{lu:graded} (together with the Borel-Casselman equivalence
\cite{borel}) relates the latter category to Iwahori-spherical
representations of a corresponding split $p$-adic group.
Together with \cite{BM1,BM2}, the inclusion in (3) may thus be
regarded as an inclusion of the real spherical unitary dual (with
``real infinitesimal character'') into the $p$-adic spherical unitary
dual, one direction of an instance of Harish-Chandra's ``Lefschetz
Principle''.  Previously Barbasch \cite{ba1,ba2} proved not only the
inclusion but in fact equality for the cases under consideration.  His
methods relied on difficult and ingenious calculations and did not
provide any hints toward the definition of functors implementing the
equalities.  In this paper, we find the functors and give a
conceptually simple proof that they preserve unitarity.

Let $G_\R$ be a real form of a reductive algebraic group $G$ and let
$K_\R$ denote a maximal compact subgroup of $G_\R$.  Using the
restricted root space decomposition of $G_\R$, one may naturally
define a graded affine Hecke algebra $\H(G_\R)$ associated to $G_\R$;
see Definition \ref{d:HGR}.  For instance, if $G_\R$ is split and connected, then
$\H(G_\R)$ is simply the equal parameter algebra associated to the
root system of $G$.  (In general the rank of $\H(G_\R)$ coincides with
the real rank $k$ of $G_\R$.)  Assume now that $G$ is one of the
classical groups $GL(V)$, $Sp(V)$, or $O(V)$.  Given a Harish-Chandra
module $X$ for $G_\R$, the Schur-Weyl duality 
of Arakawa-Suzuki \cite{AS} suggests
it is natural to look for an action of $\H(G_\R)$ on $X \otimes
V^{\otimes k}$.  In fact this space is too large to carry such an
action in general.  In the present setting, it makes sense cut down
the space in question by considering
\begin{equation}
\label{e:muintro}
\Hom_{K_\R}(\mu, X \otimes V^{\otimes k})
\end{equation}
for a representation $\mu$ of $K_\R$.  
Such spaces were considered
in \cite{CT1}, \cite{EFM}, and \cite{ma} for $GL(n,\R)$ and $U(p,q)$. 

Again, in general, this space is too large for general $\mu$.  But
for the groups $GL(n,\R),$ $U(p,q),$ $Sp(2n,\R),$ and $O(p,q)$, 
a very special choice
of $\mu = \mu_0$ allows one to define an action of $\H(G_\R)$ 
on the space in \eqref{e:muintro}.  In fact, $\mu_0$
may be characterized in an apparently completely different way as follows.
Let $P_\R = M_\R A_\R N_\R$ denote the Langlands decomposition
of a minimal parabolic subgroup of $G_\R$, and write $W_\R$ for the
Weyl group of the restricted roots of $\fra_\R$ in $\frg_\R$.  Then
the $M_\R$ fixed vectors in any representation of $K_\R$ is naturally
a representation of $W_\R$.  With this in mind, $\mu_0$ is defined
by the requirement that
\begin{equation}
\label{e:mu0intro}
\Hom_{M_\R}(\mu_0, V^{\otimes k}) = (\mu^*_0 \otimes V^{\otimes k})^{M_\R} \simeq \bbC[W_\R]
\end{equation}
as representations of $W_\R$; see Proposition \ref{p:muexistence}.  In other
words, $\mu_0$ is a necessary twist needed to construct a model of the
group algebra of $W_\R$ from the finite-dimensional representation theory of $K_\R$
(and $G_\R$).  Such a $\mu_0$ exists only for the classical real groups mentioned above;
see Remark \ref{r:quat}.  It would be interesting to find a replacement for this condition
in the remaining classical cases (and the exceptional ones too of course).

The model of $\bbC[W_\R]$ in \eqref{e:mu0intro} has a further subtle
property: $(\mu_0^* \otimes V^{\otimes k})$ is single-petaled in the
sense of Oda \cite{oda} (Definition \ref{d:oda}).  Single-petaled
representations of $K_\R$ were independently considered by Barbasch
and Vogan (e.g.~\cite{ba3}) for the purposes of matching certain
intertwining operators in the real and $p$-adic case.  Such matchings
require extensive case-by-case analysis, and are one of the main tools
for establishing inclusions of the spherical unitary dual of a real
group in a $p$-adic one (e.g.~\cite{ba1,ba2}).  But, roughly speaking,
because of \eqref{e:mu0intro} and the single-petaled property, we can
recast the intricate case-by-case matching in simple functorial
statements.

Here is a more precise formulation of our main results.  The statement
about the action of $\H(G_\R)$ is given in Corollary \ref{c:spher} (a
consequence of the more general Theorem \ref{t:class}).  Part (1) is
Theorem \ref{t:sphcomp}; parts (2) and (3) are contained
in Theorem \ref{t:herm} and Corollary \ref{c:hermcomp}.  
(Here and elsewhere if $A$ is an associative 
algebra with
unit, $A$-mod denotes the category of finite-dimensional unital left
$A$ modules.)

\begin{theorem} 
\label{t:intro}
Let $G_\R$ be one of the real groups $GL(n,\R),$
  $U(p,q)$, $Sp(2n,\R)$, or $O(p,q),$ $p\ge q.$ Let $V$ be the
  defining (``standard'') representation of $G_\R$, 
write $K_\R$ for the maximal compact
  subgroup of $G_\R$, and let $k$ denote the real rank of $G_\R$.  Let
  $\HC_\triv(G_\R)$ denote the full subcategory of Harish-Chandra modules
  for $G_\R$ whose irreducible objects are irreducible subquotients of
  spherical principal series (Definition \ref{d:HCmu}).  Let $\mu_0$
  be the character of $K_\R$ given in Proposition \ref{p:muexistence}.
  Then for any object $X$ of $\HC_\1(G_\R)$, there is a natural action
  of the graded affine Hecke algebra $\H(G_\R)$ of Definition
  \ref{d:HGR} on the space $\Hom_{K_\R}(\mu_0, X \otimes V^{\otimes
    k})$.  This defines an exact covariant functor
  \[
  F_\1 \; : \; \HC_{\1}(G_\R) \lra \H(G_\R)\text{-mod}
  \]
  with the following properties:
  \begin{enumerate}
\item  If $X_\1^\R(\nu)$ is a spherical principal series for $G_\R$ with
 parameter $\nu$ (Definition \ref{d:stdR}), then $F_\1(X^\R_\1(\nu))=X_\1(\nu)$, 
 the spherical principal series of $\H(G_\R)$ with the same parameter $\nu$
 (Definition \ref{d:stdH}).
 
\item If $X$ is an irreducible spherical representation of $G_\R$,
  then $F_\1(X)$ is an irreducible spherical representation of
  $\H(G_\R)$.

\item If, in addition to the hypotheses in (2), $X$ is Hermitian (resp.~unitary), 
then so is $F_\1(X)$.
\end{enumerate}
\end{theorem} 

We remark that our proofs are essentially self-contained.  (In Theorem
\ref{t:oda}
we do however rely on the main results of \cite{oda} to avoid some unpleasant case-by-case calculations.)

The constructions in this paper potentially apply in greater generality.  For
example, when considering nonspherical principal
series of a split group $G_\R$, it is natural to introduce a variant
of $\H(G_\R)$ built from the ``good roots'' distinguished by the
nonspherical inducing parameter (as in \cite{bcp} for instance).  
In this setting, one expects the
existence of functors from Harish-Chandra modules for $G_\R$ to modules for
this other Hecke algebra; see Remark \ref{r:endo}.  We hope to return
to this elsewhere.

\section{Hecke algebra actions}\label{s:2}

The purpose of this section is to describe natural functors
from Harish-Chandra modules for $G_\R$ to modules for a corresponding
graded affine Hecke algebra (with parameters).  
Since it contains the core ideas essential
for the rest of the paper,
we give an overview.  After some preliminaries
in Section \ref{s:2.1},
we work in Section \ref{s:2.2} in the general setting of arbitrary 
$U(\frg)$ modules $X$ and $V$,
and recall the action of (a thickening of) the Lie algebra
$\caA_k$
of the affine braid group
on $X\otimes V^{\otimes k}$ (Lemma \ref{l:Aact}).  
In Section \ref{s:2.3}, we recall the closely related
graded affine Hecke algebra $\H_k$ of type $\mathrm{A}_{k-1}$,
a quotient of $\caA_k$.  
Using the results of 
Section \ref{s:2.2} we are able to give a condition
on $V$ such that the action of $\caA_k$ on $X \otimes V^{\otimes k}$
descends to one for $\H_k$.  The condition is automatic if
$\frg =\frg\frl(V)$ (Lemma \ref{omegaclass}(a)), for instance, 
a fact used extensively in \cite{AS} and \cite{CT1},
but it is nontrivial in general.
The question is to find a natural class of modules $X$ and
a natural subspace of 
$X\otimes V^{\otimes k}$ for which the condition holds.

At this point we impose the further hypothesis that $X$ is a
Harish-Chandra module for a real classical group $G_\R$ with maximal
compact subgroup $K_\R$, that $X$ arises as a subquotient of certain
principal series with suitably ``small'' lowest $K_\R$ type, and that
$V$ is the defining representation of $G$.  We then find a natural
$K_\R$ isotypic component of $X\otimes V^{\otimes k}$ on which $\H_k$
acts (Proposition \ref{p:classHk}).  In fact, as discussed in the
introduction, the Lefschetz principle suggest that a larger Hecke
algebra $\H(G_\R)$ (containing $\H_k$ as a subalgebra) should act when
$k$ is taken to be the real rank of $G_\R$ (and certain parameters are
introduced).  The algebra $\H(G_\R)$ is defined in Definition
\ref{d:HGR}, and its action is obtained in Corollary \ref{c:spher}, a
consequence of the stronger Theorem \ref{t:class}.
A final section contains
details of an alternative presentation of $\H(G_\R)$ needed in Sections \ref{s:3} and  \ref{s:4}.

\subsection{The Casimir element and its image under comultiplication}
\label{s:2.1} 
\label{s:casimir} 
Let $\frg$ be a complex reductive Lie algebra.
Let $\kappa$ denote a fixed nondegenerate
symmetric bilinear $\ad$-invariant form on $\frg.$
Let $\Delta$
denote the comultiplication for the natural Hopf
algebra structure on $U(\g).$ 
More precisely, $\Delta: U(\g)\to U(\g)\otimes
U(\g)$ is the composition of the diagonal embedding $\g \to
\g\oplus\g$ with the natural isomorphism $U(\g\oplus\g)\cong
U(\g)\otimes U(\g).$ In particular:
\begin{align*}
\Delta(x)&=1\otimes x+x\otimes 1, &\quad \text{for all } x\in \g,\\
\Delta(xy)&=1\otimes xy+xy\otimes 1+x\otimes y+y\otimes
x,&\quad\text{for all }x,y\in\g.
\end{align*}
If $\B$ is a basis of $\frg$, let $\B^* = \{E^* \; | \; E \in \B\}$ where $E^*$ is defined by:
\begin{equation*}
\kappa(E,F^*)=\left\{\begin{aligned}&1,&\text{ if }E=F,\\&0,&\text{ otherwise.}\end{aligned}\right.
\end{equation*}
Let $C\in U(\g)$ denote the element
\begin{equation}
C=\sum_{E \in \B} EE^*.
\end{equation}
Then $C$ is  central in $U(\g)$ 
and does not depend on the choice of basis $\B$.  
By Schur's lemma, it acts by a scalar in any irreducible representation
of $\frg$.  To make certain calculations below cleaner in the classical
cases, we find it
convenient to rescale $C$ as follows.

\begin{convention}
\label{c:cas}
Suppose $\frg = \frg(V)$ is a classical Lie algebra with defining
representation $V$.  We rescale $C$ so that it acts
on $V$ by the scalar $\mathrm{rank}(\frg)$.
\end{convention}

Define the tensor $\Omega\in\g\otimes\g$ by:
\begin{equation}\label{omega}
\Omega=\frac 12\left (
\Delta(C)-1\otimes C-C\otimes 1
\right )
=\frac
12\sum_{E \in \B}(E\otimes E^*+E^*\otimes E).
\end{equation}
Note that if $\B = \B^*$, in other words if the basis $\B$ is
closed under taking $\kappa$-duals, 
then $\Omega=\sum_{E} E\otimes E^*.$  It is
always possible to find a such basis, for example by using a root
decomposition of $\g.$  For simplicity of notation (but without loss of
generality), we will assume this
from now on, and write $\Omega=\sum_{E} E\otimes E^*.$

It is clear that the tensor $\Omega$ is symmetric, meaning that
$R_{12}\circ \Omega=\Omega,$ where $R_{12}:\g\otimes\g\to\g\otimes\g$
is the flip $R_{12}(x\otimes y)=y\otimes x.$ As a consequence of the
fact that $C$ is a central element, one can verify that $\Omega$ is
$\g$-invariant,
\begin{equation}
\label{ginv}
[\Delta(x),\Omega]=0,\quad\text{for all } x\in \g.
\end{equation}

\subsection{The algebra $\caA_k$}
\label{s:2.2} 
\label{s:qa}
Algebraic constructions related to the material in this subsection may
be found in \cite{Ka}, Chapters XVII and XIX, for instance.

Let $k\ge 2$ be fixed, and fix a $\kappa$ self-dual basis $\B$
of $\frg$ as above. For every $0\le i\neq j\le k$,
define a tensor in $U(\g)^{\otimes (k+1)}$:
\begin{equation}\label{tens}
\Omega_{i,j}= \sum_{E \in \B} (E)_i\otimes (E^*)_j,
\end{equation}
with notation as follows.  Given $x \in U(\frg)$ and $0\leq l \leq k$,
$(x)_l \in U(\frg)^{\otimes k+1}$ denotes the simple tensor with $x$ in the
$(l+1)$st position and $1$'s in the remaining positions.

\begin{lemma}\label{infbraid} In $U(\g)^{\otimes(k+1)}$, we have the following
  identities:

\begin{enumerate}
\item[(1)] $[\Omega_{i,j},\Omega_{m,l}]=0,$ for all distinct
  $i,j,m,l$.
\item[(2)] $[\Omega_{i,j},\Omega_{i,m}+\Omega_{j,m}]=0,$ for all
  distinct $i,j,m.$  
\end{enumerate} 

\end{lemma}

\begin{proof} (1) is clear. For (2),  we
  have:

\begin{align*}
\biggr[ \sum_E (E)_i\otimes (E^*)_j\otimes (1) \; \; ,\; \; &\sum_F (F)_i\otimes
(1)_j\otimes (F^*)_m+(1)_i\otimes (F)_j\otimes (F^*)_m \biggr ]\\
&=\sum_F \biggr[\sum_E
(E)_i\otimes (E^*)_j, (F)_i\otimes (1)_j+(1)_i\otimes (F)_j\biggr ]\otimes
(F^*)_m\\
&=\sum_F[\Omega,\Delta(F)]_{i,j}\otimes (F^*)_m=0,
\end{align*}
by the $\g$-invariance \eqref{ginv} of $\Omega.$

\end{proof}

It is convenient to package the relations of Lemma \ref{infbraid}
together with various permutations into an abstract algebra.
Let $S_k$ be the symmetric group in $k$ letters, and let
$s_{i,j}$ the transposition $(i,j)$, for $1\le i<j\le k.$

\begin{definition}\label{Ak} Le $\caA_k$ to be the complex
  associative algebra with unit generated by $S_k$ and elements $\omega_{i,j},$
  $0\le i\neq j\le k,$ subject to the relations:
\begin{enumerate}
\item[(1)] $[\omega_{i,j},\omega_{m,l}]=0,$ for all distinct
  $i,j,m,l$.
\item[(2)] $[\omega_{i,j},\omega_{i,m}+\omega_{j,m}]=0,$ for all
  distinct $i,j,m.$  
\item[(3)] $s_{i,j}\omega_{i,l}=\omega_{j,l}s_{i,j}$, for all
  distinct $i,j,l$ and $s_{i,j}\omega_{l,m}=\omega_{l,m}s_{i,j}$, for
  all distinct $i,j,l,m.$
\end{enumerate}

\end{definition}

\bigskip

Next let $X$ be a $U(\g)$-module, and assume that $V$ is 
a representation of $\frg$. 
Denote the action of $U(\frg)^{k+1}$
on $X \otimes V^{\otimes k}$
by $\pi_k$.
By a slight abuse of notation, also let $\pi_k$
denote the signed action of $S_k$
permuting the factors in $V^{\otimes k}$,
\begin{equation}
\label{e:Sact}
\pi_k(\sigma):\  x\otimes v_1\otimes\dots\otimes v_k\mapsto \sgn(\sigma) (x\otimes
v_{\sigma^{-1}(1)}\otimes\dots\otimes v_{\sigma^{-1}(k)}), \quad
\sigma\in S_k.
\end{equation}

\begin{lemma} 
\label{l:Aact}
With the notation and definitions as above, there is a natural action
of $\caA_k$ on $X \otimes V^{\otimes k}$ defined on generators by
\begin{align*}
\sigma &\mapsto \pi_k(\sigma) \\
\omega_{i,j} &\mapsto \pi_k(\Omega_{i,j})
\end{align*}
This action commutes with the action of
  $U(\g)$ on $X\otimes V^{\otimes k}$.
\end{lemma}

\begin{proof}
An easy verification shows that, as operators on $X\otimes V^{\otimes k}$,
$\pi_k(\Omega_{i,j})$ and
$\pi_k(\sigma)$ 
satisfy the commutation relations of Definition
\ref{Ak}.  This is the first part of the lemma.

It remains to check that the two actions commute.
The action of $x \in \frg$ is of course given by $\sum_{l=0}^k \pi_k((x)_l)$,
and it is clear that this commutes with the action of $\sigma$ for $\sigma \in S_k$.  
To check the it commutes with the action of $\omega_{i,j}$ we verify
\begin{align*}
\left [ \pi_k (\Omega_{i,j})\; , \; \sum_{l=0}^k \pi_k((x)_l) \right ] &=
\biggr [\pi_k \bigr ( \sum_E (E)_i\otimes
(E^*)_j  \bigr ) \; \;  , \; \; \pi_k \bigr (\sum_{l=0}^{k+1}(x)_l \bigr ) \biggr] \\
&= 
\pi_k \biggr [\sum_E (E)_i\otimes
(E^*)_j\; \; , \; \; (x)_i\otimes (1)_j+(1)_i\otimes(x)_j \biggr]
\\&=\pi_k \bigr( [\Omega,\Delta(x)]  _{i,j} \bigr )=0,
\end{align*}
by \eqref{ginv}.
\end{proof}

\subsection{The Type A affine graded Hecke algebra 
and its action on $X \otimes V^{\otimes k}$ for $\frg  = \frg\frl(V)$} 
\label{s:2.3}

\begin{definition} 
\label{d:glngaha}
The algebra $\H_k$ for $\frg\frl(k)$ is the complex associative algebra
with unit generated by $S_k$ and
  $\ep_l,$ $1\le l\le k$, subject to the relations:
\begin{enumerate}
\item[$\bullet$] $[\ep_l,\ep_m]=0,$ for all $1\le l,m\le k.$
\item[$\bullet$] $s_{i,j}\ep_l=\ep_l s_{i,j}$, for all distinct $i,j,l.$
\item[$\bullet$] $s_{i,i+1}\ep_i-\ep_{i+1} s_{i,i+1}=1,$ $1\le i\le
  k-1.$
\end{enumerate}
\end{definition}
\bigskip

\noindent
We next investigate when it is possible to
define an action of $\H_k$ on $X\otimes V^{\otimes k}$ 
from the action of $\caA_k$. Set
\begin{equation}
\vep_l=\omega_{0,l}+\omega_{1,l}+\dots+\omega_{l-1,l} \in \caA_k.
\end{equation}

\begin{lemma}\label{l:Akact}In $\caA_k$, we have 
\begin{enumerate}  
\item[(1)] $[\vep_l,\vep_m]=0,$ for all $l,m.$
\item[(2)] $s_{i,j}\vep_l=\vep_l s_{i,j}$, for all distinct $i,j,l.$
\item[(3)]   $s_{i,i+1}\vep_i-\vep_{i+1}s_{i,i+1}=-\omega_{i,i+1}s_{i,i+1}.$
\end{enumerate}
\end{lemma}

\begin{proof}  To prove (1), we have for $l<m$:
\begin{align*}
[\vep_l,\vep_m]=\sum_{0\le i<l}\left [\omega_{i,l}\; \; ,\; \; \sum_{0\le j<m}\omega_{j,m}\right ]
             =\sum_{0\le i<l}[\omega_{i,l},\omega_{i,m}+\omega_{l,m}]=0,
\end{align*}
where we have used the first two defining relations of Definition \ref{Ak}.
The assertion in (2) is obvious.  To prove (3), we use the third defining relation repeatedly.
\end{proof}

Because we are ultimately interested in the defining an action of $\H_k$, the the right-hand side of
the equality in Lemma \ref{l:Akact}(2) makes it natural to examine the action of $\Omega_{i,i+1}$ on $X\otimes V^{\otimes k}.$ This comes down to computing the action of $\Omega$ on $V\otimes V.$  The Casimir element $C$ acts by a scalar $\chi_U(C)$ on each irreducible $U(\g)$-module $U.$ 
Therefore, from (\ref{omega}), one sees that, on $V\otimes V$, we
have
\begin{equation}\label{omegaact}
\pi_2(\Omega) =\frac 12 \bigoplus_U \chi_U(C)\pr_U -\frac
12\bigoplus_{U'}(1\otimes\chi_{U'}(C)\pr_{U'}+\chi_{U'}(C)\pr_{U'}\otimes
1),
\end{equation}
where $V=\bigoplus U'$ and $V\otimes V=\bigoplus U$ are the decompositions into
irreducible $U(\g)$-modules, and $\pr_U$, $\pr_{U'}$ denote the
corresponding projections.   
In the special case that $\frg$ is classical,
\eqref{omegaact} can be made very explicit. 

\begin{lemma}\label{omegaclass} Let $\frg$ be a classical Lie algebra
and let $V$ denote its defining representation.  Rescale the Casimir
element $C$ as in Convention \ref{c:cas}.  Recall
the flip operation $R_{12}(x \otimes y) = y \otimes x$ and the element
$\Omega \in \frg \otimes \frg$ defined in  \eqref{omega}.
Then
\begin{enumerate}
\item If $\g=\frg\frl(V)$, $\pi_2(\Omega)=R_{12}$ as operators on $V \otimes V$.

\item If $\g=\frs\frp(V)$ or $\frs\fro(V)$,
$\pi_2(\Omega)=R_{12}-(\dim V) \pr_\1$ as operators on $V\otimes V$; 
here $\pr_\1$ denotes the projection onto the trivial
 $U(\frg)$ isotypic component of $V\otimes V.$ 
\end{enumerate}
\end{lemma}

\begin{proof}
The decomposition of $V \otimes V$ is well-known, and using \eqref{omegaact}
we conclude, as an operator on $V \otimes V$,
\begin{align*}
&\pi_2(\Omega)=\frac 12 \chi_{S^2V}(C)\pr_{S^2V}+\frac
12\chi_{\wedge^2V}(C)\pr_{\wedge^2V}-\chi_V(C), \text{ for }
\frg\frl(V),\\
&\pi_2(\Omega)=\frac 12 \chi_{S^2V/\1}(C)\pr_{S^2V/\1}+\frac
12\chi_{\wedge^2V}(C)\pr_{\wedge^2V}+\frac
12\chi_{\1}(C)\pr_\1-\chi_V(C), \text{ for }
\frs\fro(V),\\
&\pi_2(\Omega)=\frac 12 \chi_{S^2V}(C)\pr_{S^2V}+\frac
12\chi_{\wedge^2V/\1}(C)\pr_{\wedge^2V/\1}+\frac
12\chi_{\C}(C)\pr_\1-\chi_V(C), \text{ for }
\frs\frp(V).
\end{align*}
Note also that $R_{12}=\pr_{S^2V}-\pr_{\wedge^2V}.$ The scalars by
which $C$ acts are easily computable (on highest weight spaces, for instance). 
The lemma follows.
\end{proof}

\noindent
The following appears as Theorem 2.2.2 in \cite{AS}.
\begin{cor}
\label{c:typeA}
Suppose $\frg = \frg\frl(V)$ and $X$ is any $U(\frg)$ module.  Then there is a natural
action of $\H_k$  on $X \otimes V^{\otimes k}$ defined on generators by
\begin{align*}
\epsilon_l &\mapsto \pi_k\left (\Omega_{0,l}+\Omega_{1,l}+\dots+\Omega_{l-1,l}\right ) \\
s_{i,i+1} &\mapsto \pi_k(s_{i,i+1})
\end{align*}
\end{cor}

\begin{proof}
This follows immediately from  Lemma \ref{l:Aact},  Lemma \ref{l:Akact}(2), and 
Lemma \ref{omegaclass}(1).
\end{proof}
\medskip

Lemma \ref{omegaclass}(2) (and more generally \eqref{omegaact}) show how the corollary
can fail for general $\frg$.  In the next section, for other
classical algebras, we find a natural subspace
of $X \otimes V^{\otimes k}$ and a natural class of modules $X$
on which $\H_k$ indeed acts.

\subsection{Action of $\H_k$ on subspaces of $X \otimes V^{\otimes k}$ for Harish-Chandra modules for classical $\frg$.}
\label{s:realform} 
Begin by assuming $G$ is a complex reductive algebraic group and $G_\R$ is a real form
with Cartan involution $\theta$.  Write $K_\R$ for the maximal compact subgroup 
consisting of the fixed points of $\theta$ on $G_\R$, and $K$ for its complexification in $G$.  
Assume that $X$ and $V$ are objects in the category $\HC(G_\R)$ of
Harish-Chandra modules for $G_\R$ with $V$ finite-dimensional. 
From the $G$-invariance of $\kappa$, and in particular
$K_\R$-invariance, we see that the action of operators $\Omega_{i,j}$ defined in
(\ref{tens}) on $X \otimes V^{\otimes k}$ 
commutes with the diagonal action of $K_\R.$ This implies
that for every representation $(\mu,V_\mu)$ of $K_\R$, we obtain an
exact functor
\begin{equation}\label{functAk}
\wt F_{\mu,k,V}: ~\caH(G_\R)\longrightarrow \caA_k\text{-mod},\quad 
\wt F_{\mu,k,V}(X):=\Hom_{K_\R}[V_\mu, X\otimes V^{\otimes k}].
\end{equation}
When $\frg = \frg\frl(n)$ (or $\frs\frl(n)$), Corollary \ref{c:typeA}
shows this functor this functor descends to one which has the image
in the category $\H_k$-mod. 
By making a judicious choice of $\mu$, we seek to make the same conclusion
for a natural class of modules for other groups outside of Type A.

We need some more notation.  Let $\g_\R=\frk_\R+\frp_\R$ be the Cartan
decomposition of the Lie algebra of $G_\R$, and fix a maximal abelian
subspace $\fra_\R$ of $\frp_\R.$ (To denote the corresponding
complexified algebras, we drop the subscript $\R.$) Fix a maximal
abelian subspace $\fra_\R$ of $\frp_\R$, let $k$ denote the dimension
of $\fra_\R$ (i.e.~the real rank of $G_\R$), and let $M_\R$ denote the
centralizer of $\fra_\R$ in $K_\R$.  Let $\Phi$ denote the system of
(potentially nonreduced) restricted roots of $\fra_\R$ in $\frg_\R$.
We thus have a decomposition $\frg_\R= \frm_\R \oplus
\fra_\R\oplus\bigoplus_{\al\in\Phi}\frg_\al,$ where $\frg_\al$ is the
root space for $\al$ and $\frm_\R$ is the Lie algebra of $M_\R$.
Write $W_\R$ for the Weyl group of $\Phi$,
\begin{equation}
\label{WR}
W_\R = N_{K_\R}(A_\R)/M_\R.
\end{equation}
We fix once and for all a choice of simple roots $\Pi_\circ$ in the reduced part
$\Phi_\circ$ of $\Phi$, and let $M_\R A_\R N_\R$ denote the corresponding
minimal parabolic subgroup of $G_\R$.

We next introduced a restricted class of Harish-Chandra 
modules on which we will eventually define our functors.
\begin{defn}
\label{d:HCmu}
\label{d:stdR}
Suppose $\delta$ is a one-dimensional representation of $K_\R$ 
(and, by restriction, $M_\R$).   
Fix $\nu \in \fra^*$ and assume $\nu$
is dominant with respect to the roots of $\fra_\R$ in $\frn_\R$.   
Let $X^\R_\delta(\nu)$ denote the
minimal principal series $\Ind_{M_\R A_\R N_\R}(\delta \otimes e^\nu \otimes \1)$ where the induction
is normalized as in \cite[Chapter VII]{knapp}. 
Let $\ol{X}^\R_\delta(\nu)$ denote the unique irreducible subquotient of $X_\delta^\R$ containing the $K_\R$
representation $\delta$.

Define the full subcategory
$\HC_\delta(G_\R$) of $\HC(G_\R)$ 
to consist of objects which are subquotients of the various
$X^\R_\delta(\nu)$ as $\nu$ ranges over  all (suitably dominant) elements $\fra^*$.  
\end{defn}

\begin{remark}
\label{r:endo}
Since our main results in Sections \ref{s:3} and \ref{s:4}
are about spherical representations, the setting
of Definition \ref{d:HCmu} is entirely appropriate.  (Note also that in the case
of $\delta = \1$, $\HC_\1(G_\R)$ is a real analogue of the category of
Iwahori-spherical representation of a split $p$-adic group \cite{borel}.)
For applications
to the nonspherical case (as mentioned in the introduction), 
the natural setting is to assume $G_\R$
is quasisplit and the $K_\R$-type $\delta$ is fine in the sense
of Vogan.
\end{remark}

Our next task is to find the appropriate $\mu$ and $V$ so that $\wt
F_{\mu, k, V}(X)$ (for $X$ in some $\HC_\delta(G_\R)$) has a chance of
carrying an $\H_k$ action.  Given any representation $U$ of $K_\R$,
the $M_\R$ fixed vectors $U^{M_\R}$ are naturally a representation of
$W_\R$.  In the proofs in Sections \ref{s:3} and \ref{s:4} below, we
need the existence of a representation $V$ of $G_\R$ and a character
$\mu_0$ of $K_\R$ such that
\begin{equation}
\label{e:mudef}
(\mu^*_0 \otimes V^{\otimes k})^{M_\R} \simeq \bbC[W_\R]
\end{equation}
as $W_\R$ representations.
This is admittedly a rather mysterious
requirement at this point, but it emerges naturally (and as explained in the introduction
is
related to results of Barbasch and Oda).  So we are left to
investigate it.

\begin{prop}
\label{p:muexistence}
Suppose $G_\R = GL(n,\R), U(p,q), Sp(2n,\R)$ or $O(p,q)$.  Let $V$
denote the defining representation of $G_\R$.  Then there exists
a character $\mu_0$ satisfying \eqref{e:mudef}.
Explicitly, we have
\begin{enumerate}
\item[(i)] $G_\R = GL(n,\R)$.  Then $\mu_0$ is the nontrivial (``sign of the determinant'')
character $\sgn$ of $K_\R = O(n)$.

\item[(ii)] $G_\R = U(p,q), p \geq q$.  Then $\mu_0$ is the character $\1 \boxtimes \det$
of $K_\R \simeq U(p) \times U(q)$.

\item[(iii)] $G_\R = Sp(2n,\R)$.  Then $\mu_0$ is the determinant character $\det$
of $K_\R \simeq U(n)$.

\item[(iv)] $G_\R = O(p,q), p \geq q$.  Then $\mu_0$ is the character $\1 \boxtimes \sgn$
of $K_\R \simeq O(p) \times O(q)$.
\end{enumerate}
\end{prop}

\begin{remark}
\label{r:quat}
No such character $\mu_0$ as in Proposition \ref{p:muexistence}
exists for the quaternionic
series of classical groups $GL(n,\bbH)$, $Sp(p,q),$ and $O^*(2n)$.  
\end{remark}

\begin{proposition}\label{p:classHk} 
Let $G_\R$ be one of the groups
$GL(n,\R)$, $U(p,q),$ $Sp(2n,\R),$ $O(p,q)$, $p\ge q,$ and $V$ be
the defining representation of $G_\R$. 
Let $k$ be the real rank of $G_\R.$
Fix a one-dimensional representation $\mu$ of $K_\R$ and fix $\mu_0$ as
in Proposition \ref{p:muexistence}.  Recall the functor $\wt F_{\mu,k,V}$
from \eqref{functAk}.  Then for each object 
$X \in \HC_{\mu \otimes \mu^*_0}(G_\R)$, 
there is a natural $\H_k$ action on 
$\wt F_{\mu,k,V}(X)$ defined on generators by
\begin{align*}
\epsilon_l &\mapsto \text{composition with } \pi_k\left (\Omega_{0,l}+\Omega_{1,l}+\dots+\Omega_{l-1,l}\right ) \\
s_{i,i+1} &\mapsto \text{composition with } \pi_k(s_{i,i+1})
\end{align*}
(cf.~Corollary \ref{c:typeA}).
Hence we obtain an exact functor
\begin{align*}
F_{\mu,k,V}\; : \; \HC_{\mu \otimes \mu^*_0}(G_\R) &\longrightarrow \H_k\text{-mod} \\
X &\longrightarrow \Hom_{K_\R} ( \mu, X \otimes V^{\otimes k} ).
\end{align*}

\end{proposition}

\noindent
Before turning the proofs, we remark outside of type A,
the definition on generators given in Proposition \ref{p:classHk} fails to extend to an action 
of $\H_k$ if $k$ is not taken to be the real rank of $G_\R$.

\subsection{Proof of Propositions \ref{p:muexistence} and \ref{p:classHk}}
\label{s:proof}
We begin with a lemma whose proof constructs an explicit basis
for the image of $\wt F_{\mu,m,V}$.  The proof also shows the importance
of choosing $k$ to be the real rank of $G_\R$.

\begin{lemma}\label{l:hom} 
Retain the setting of Proposition \ref{p:classHk}. 
\begin{enumerate}
\item If $m<k$ then $\wt F_{\mu,m,V}(X^\R_{\mu\otimes\mu^*_0}(\nu))=0.$ 
\item
  If
  $m=k,$ then $\dim\wt F_{\mu,m,V}(X^\R_{\mu \otimes\mu^*_0}(\nu))=|W_\R|.$
\end{enumerate}
\end{lemma}

\begin{proof}
For (1), we use Frobenius reciprocity:
\begin{align}
\label{e:frob}
\wt F_{\mu,k,V}(X^\R_{\mu\otimes \mu^*_0}(\nu))&= 
\Hom_{K_\R}[\mu, X^\R_{\mu\otimes \mu^*_0}(\nu)\otimes V^{\otimes
    m}]\\
&=\Hom_{M_\R}[\triv,\mu^*_0 \otimes V^{\otimes m}]
=(\mu^*_0 \otimes  V^{\otimes m})^{M_\R}.
\end{align}
We describe bases for this latter vector space corresponding to some explicit
realizations of $G_\R$ and $K_\R$. 
Let $e_1,e_2\dots,$ denote the standard basis of $V$ used to realize
$G_\R$ as matrices.

\begin{enumerate}

\item[(a)] $GL(n,\R).$ We may arrange the group
  $M_\R=O(1)^n$ to be diagonal. We have
  $$V|_{M_\R}=\bigoplus_{i=1}^n\triv\boxtimes\dots\boxtimes \underbrace{\sgn}_{i\text{th}}\boxtimes\dots\boxtimes\triv,$$
where $\triv\boxtimes\dots\boxtimes
\underbrace{\sgn}_{i\text{th}}\boxtimes\dots\boxtimes\triv$ is
generated by $e_i.$ If $m<n,$ then there can't be any copies of
$\mu_0|_{M_\R}=\sgn\boxtimes\dots\boxtimes\sgn$ in $V^{\otimes m}.$  When $m=n$, the basis vectors are
  $$\{e_{\sigma(1)}\otimes\dots \otimes e_{\sigma(n)}:~ \sigma\in S_n\}.$$
\end{enumerate}

For the other groups, the analysis is similar, so we only give the
realizations and the basis elements when $m=k.$

\begin{enumerate}
\item[(b)] $U(p,q), p\ge q.$ 
Let the form $J$ defining $U(p,q)$ be
given (in the basis $\{e_i\}$) by 
  the diagonal matrix 
$(\underbrace{1,\dots,1}_p,\underbrace{-1,\dots,-1}_q)$. 
Since the elements of $M_\R$ are block diagonal matrices of the form
$(x;x_1,\dots,x_q,x_q,\dots,x_1),$ with $x\in U(p-q)$ and
$x_1,\dots,x_q\in U(1),$ the space $(\mu^*_0\otimes V^{\otimes
  q})^{M_\R}$ is spanned by 
\begin{equation*}
\{f_{\sigma(1)}^{\eta_1}\otimes\dots\otimes
f_{\sigma(q)}^{\eta_q}: ~\eta_j\in\{\pm 1\},\ \sigma\in S_q\},
\end{equation*} 
where $ f_j^\eta=e_{p-j+1}+\eta e_{p+j},$ $1\le j\le q.$

\item[(c)] $O(p,q), p\ge q.$  
Let the form $J$ defining $O(p,q)$ once again be
given by 
  the diagonal matrix 
$(\underbrace{1,\dots,1}_p,\underbrace{-1,\dots,-1}_q)$. 
Thus we naturally
realize $O(p,q)$ as a subgroup
of the realization $U(p,q)$ given in (b).
So we can choose the same basis elements for $(\mu^*_0\otimes
V^{\otimes q})^{M_\R}.$

\item[(d)] $Sp(2n,\R).$  We choose the form
$J=\begin{pmatrix}0_n&J_n\\-J_n&0_n\end{pmatrix}$ where $J_n$
is matrix with 1's along the antidiagonal and zeros everywhere
else.  This realization of 
$Sp(2n,\R)$ is naturally a subgroup of the realization of
$U(n,n)$ given in (b).  So we may once again
choose the same basis elements.

\end{enumerate}
\end{proof}

\begin{proof}[Proof of Proposition \ref{p:muexistence}]
Using the explicit bases constructed in the proof of Lemma \ref{l:hom},
the proposition becomes very easy to verify.  We omit the details.
\end{proof}

\begin{proof}[Proof of Proposition \ref{p:classHk}]
Because of Lemma \ref{l:Aact}, Lemma \ref{l:Akact}(3), and the relations defining $\H_k$,
the proposition reduces to showing that composition with $\pi_k(s_{i,i+1})$ coincides with
composition with $\pi_k(\Omega_{i,i+1})$ on  $\Hom_{K_\R} ( \mu, X \otimes V^{\otimes k} )$.
Because of Corollary \ref{c:typeA}, we may assume $\frg$ is not of type A.
Let $\bf J$ denote the quadratic form on $V$ corresponding to $J$ defining $G_\R$
as in the proof of Lemma
\ref{l:hom}.
In light of Lemma \ref{omegaclass}(b) and the proof of Lemma \ref{l:hom},
we need only check for each $i$ that the kernel of the
projection $\pr_i$ from $V^{\otimes k}$ to $V^{\otimes(k-2)}$ defined by
\[
\pr_i(v_1\otimes \cdots \otimes v_k) = \mathbf{J}(v_i,v_{i+1})(v_1 \otimes\cdots\otimes v_{i-1}
\otimes v_{i+2} \otimes \cdots \otimes v_k)
\]
contains $(\mu^*_0 \otimes V^{\otimes k})^{M_\R}$.  But this is a simple
verification using the explicit bases given in the proof of Lemma
\ref{l:hom}.
\end{proof}

%

\subsection{The graded affine Hecke algebra attached to 
$G_\R$.}
\label{s:HGR}
We first recall Lusztig's definition of the general affine graded 
Hecke algebra with parameters, and then distinguish certain
parameters using the real group $G_\R$.  An alternative
presentation will be given in Section \ref{s:drin}.  

\begin{definition}[{\cite{lu:graded}}]\label{d:gaha} 
Let $R$ be a root system in a complex vector space $V$.
(We do not assume $R$ spans $V$.)
Denote the action of the Weyl group $W$ of $R$ on $V^*$ by $w\cdot f$
for $w\in W$ and $f \in V^*$.
Set $\Psi=(R, V)$ and let $\co:R\to \Z$ be a $W$-invariant function.
  The affine graded
  Hecke algebra $\H = \H(\Psi,\co)$ is 
 the unique complex associative algebra on the vector space $S(V^*) \otimes
\C[W]$
such that
\begin{enumerate}
\item[(1)] The map $S(V^*) \ra \H$ sending $f$ to $f \otimes 1$ is an algebra
homomorphism.

\item[(2)] The map $\C[W] \ra \H$ sending $w$ to $1 \otimes w$ is an algebra
homomorphism.

\item[(3)] In $\H$, we have $(f \otimes 1)(1 \otimes w) = f \otimes w.$

\item[(4)] For each simple root $\alpha$ in a fixed choice of simple roots
  $S \subset R$ and for each $f \in V^*$,
\[
(1 \otimes s_\alpha)(f\otimes 1) -[s_\alpha \cdot f]\otimes s_\alpha 
=\co(\al) f(\al).
\]
here $s_\alpha \in \C[W]$ is the reflection corresponding to $\alpha$.
\end{enumerate}
The choice of $S$ does not affect the isomorphism class of
$\H(\Psi,\co)$.  As usual, we identify $S(V^*)$ and $\C[W]$ with their
images under the maps (1) and (2), and write (4) as
\begin{equation}
s_\alpha f - (s_\alpha\cdot f)s_\alpha = \co(\al) f(\alpha).
\end{equation}

\end{definition}

\begin{remark}
For the set of roots $R_k\subset V_k:=\frh_k^*$ of a Cartan subalgebra $\frh_k$
of $\frg\frl(k,\C)$ and constant parameters $\co\equiv 1$, the algebra
$\H((R_k,V_k),\co)$ of Definition \ref{d:gaha} coincides
with the algebra $\H_k$ of Definition \ref{d:glngaha}.
\end{remark}

\begin{example}  
\label{e:tHk}
Suppose $\Psi = (R,V)$ is of type $C_k$.  Explicitly
take $R=\{\pm 2e_i \; | \; 1\leq i \leq k\} \cup \{(\pm(e_i \pm e_j) \; | 1\leq i < j \leq k\}$,
$V = \text{span}(e_1, \dots, e_k)$, and choose $S =\{e_1- e_2, \dots, e_{k-1} - e_k, 2e_k\}$.
Let $\co$ be
$1$ on short roots and a fixed value $c \in \bbC$ on long roots.  
We let
$\wt \H_k(c)$ denote the algebra $\H(\Psi, \co)$ defined by Definition \ref{d:gaha}.
If we write $\{\eps_1,\dots,\eps_k\} \subset V^*$ for the basis dual to $\{e_1, \dots, e_k\}$,
$\wt \H_k(c)$ is
generated by the simple reflections $s_{i,i+1}, 1\le
i<k-1$ in the short simple roots, the reflection 
$\bar s_k$ in the long simple root, 
and $\{\eps_1,\dots,\eps_k\}$ with the commutation relations:
\begin{equation}\label{Bk}
\begin{aligned}
&s_{i,i+1}\eps_j-\eps_j s_{i,i+1}=0, \qquad j\neq i,i+1; \\
&s_{i,i+1}\eps_i-\eps_{i+1}s_{i,i+1}=1;\\
&\bar s_k\eps_j-\eps_j\bar s_k=0,\qquad\qquad \quad  j\neq k;\\
&\bar s_k\eps_k+\eps_k\bar s_k=2c.
\end{aligned}
\end{equation}
Clearly $\H_k$ of Definition \ref{d:glngaha} is naturally a subalgebra of $\wt \H_k(c).$
\end{example}

\medskip

Return to the setting of an arbitrary real reductive group $G_\R$
(as considered at the beginning of Section \ref{s:realform}).
Recall $\Phi$ denotes the roots of $\fra_\R$ in $\frg_\R$
and $\Phi_\circ$ denotes those $\alpha$ such
that $2\alpha$ is not such a root. 
For every $\al\in \Phi_\circ,$ set
\begin{equation}\label{cal}
\co(\alpha) = \dim(\frg_\R)_\al+2 \dim(\frg_\R)_{2\al},
\end{equation} 
the sum of the (real) dimensions of the generalized $\alpha$
and $2\alpha$ eigenspaces of $\ad(\fra_\R)$ in $\frg_\R$.
Then $\co$ is constant on $W_\R$
orbits.
As usual, we let $\fra$ denote
the complexification of $\fra_\R$.

\begin{definition}[{cf.~\cite[Section 4]{oda}}]
\label{d:HGR}
In the setting of the previous paragraph, suppose in addition that
$G_\R$ is connected.  Then define  $\H(G_\R)$ 
to be the algebra $\H(\Psi, \co)$ attached by Definition \ref{d:gaha} 
to $\Psi= (\Phi_\circ, \fra^*)$ and $\co$ of \eqref{cal}.

For disconnected groups
the correct definition of $\H(G_\R)$ incorporates
the action of the component group of $G_\R$ on $\H(\Psi,\co)$.  
(The reason that extra care is required is because we are aiming
for statements like Theorem \ref{t:intro}(3). Since the action
of the component
group of $G_\R$ affects the notion of being a Hermitian
representation of $G_\R$, we need to balance this effect on the
Hecke algebra side by also 
incorporating the action of the component group there.)
We make
this explicit in the next example for $GL(n,\R)$ and $O(p,q)$, the only
disconnected cases of interest to us here.
\end{definition}

\begin{example}
\label{e:HGR} 
For particular $G_\R$, Table \ref{t:HGR} lists 
the restricted root system $\Phi$, its reduced part $\Phi_\circ$, the values
$\co(\al)$ of \eqref{cal}, and $\H(G_\R)$.  
In each case, the latter algebra
is isomorphic to one of the form
$\H_k$ or $\wt \H_k(c)$ as considered in Example \ref{e:tHk}.
As remarked in Definition \ref{d:HGR}, Table \ref{t:HGR}
{\em defines} $\H(G_\R)$ for the disconnected groups under
consideration.  For $GL(n,\R)$ and $O(p,q)$ with $p \neq q$, we simply take
$\H(G_\R) = \H(\Psi, \co)$ as in Definition \ref{d:HGR}.  (The relevant action
of the component group is trivial in these cases.)
But if $p = q$, there is a natural
identification
$$\wt \H_q(0)  \simeq 
  \H(D_q,1)\ltimes \Z/2\Z$$ of the graded Hecke algebra of type $B_q$
  with parameters $\co(\alpha_{\longg})=1, \co(\alpha_{\mathrm{short}})=0$,
  with the semidirect product of $\Z/2\Z$ with the
  graded Hecke algebra of type $D_q$ with parameters
  $\co\equiv 1.$  Here the finite group $\Z/2\Z$ acts in the
  usual way (by permuting the roots $\eps_{k-1}\pm\eps_k$) on the
  Dynkin diagram of type $D_k.$  Thus, in the setting of Definition
\ref{d:HGR} for $O(q,q)$, $\H(\Psi,\co)$ is naturally
an index two subalgebra of $\wt \H_q(0)$, and we define
$\H(O(q,q)) = \wt \H_q(0)$.

\begin{table}[h]
\caption{Examples of $\H(G_\R)$ for various classical groups. \label{t:HGR}}
\begin{tabular}{|c|c|c|c|c|c|}
\hline
$G$ &$G_\R$ &$\Phi$& $\Phi_\circ$ &$\co(\al)$ & $\H(G_\R)$ \\
\hline
$GL(n,\C)$ &$GL(n,\R)$ &$A_{n-1}$ & $A_{n-1}$ & $\co \equiv 1$
& \phantom{$\bigoplus$} $\H_n$  \\\hline
 &$U(q,q)$
&$C_q$ & $C_{q}$ &$\co(\al_\short)=2$
& \phantom{$\wt \bigoplus$} $\wt \H_q(1/2)$
\\
&&&&$\co(\al_\longg)=1$&
\\
\hline
 &$U(p,q)$, $p > q$
&$BC_q$ & $B_{q}$
&$\co(\al_\short)=1$
&  \phantom{$\wt \bigoplus$} $\wt \H_q((p-q+1)/2)$
\\
&&&&$\co(\al_\longg)=p-q+1$&\\ \hline 
$Sp(2q,\C)$ &$Sp(2q,\R)$ &$C_q$ & $C_q$
&$\co \equiv 1$ &  \phantom{$\wt \bigoplus$} 
$\wt \H_q(1)$\\\hline
$O(n,\C)$ &$O(p,q)$, $p > q$
&$B_q$ & $B_{q}$
&$\co(\al_\short)=1$
&  \phantom{$\wt \bigoplus$} 
$\wt \H_q((p-q)/2)$ \\
&&&&$\co(\al_\longg)=p-q$&\\
\hline
 &$O(q,q)$
&$D_{q}$ & $D_q$
&$\co \equiv 1$ &  \phantom{$\wt \bigoplus$} $\wt \H_q(0)$\\\hline
\end{tabular}
\end{table}
\end{example}

\subsection{The action of $\wt \H(c)$}
\label{s:heckeB}
Given a character $\mu$ of $K$ and $X \in \HC_{\mu\otimes \mu^*_0}(G_\R)$, 
Proposition \ref{p:classHk} defined an action of
$\H_k$ on $\Hom_{K_\R}\left ( \mu, X \otimes V^{\otimes k} \right)$.  
When $\mu\otimes \mu_0^*= \1$, the
Lefschetz principle discussed in the introduction suggests
that the algebras listed in the last column of the table in
Example \ref{e:HGR} should act.  We prove this in Corollary \ref{c:spher}
below.    Initially however we work in the setting of an arbitrary
scalar $K_\R$ type $\mu$ and investigate when $\wt \H_k(c)$ (for arbitrary
$c$) acts.  The main result is Theorem \ref{t:class} (from which Corollary 
\ref{c:spher} follows trivially).

We assume that $G_\R$ is one of the equal rank groups $U(p,q)$,
$Sp(2n,\R)$, or $O(p,q)$.  Let $\sigma$ denote the holomorphic
involution of $G$ corresponding to $G_\R$.  (The fixed points
of $\sigma$ are thus the complexification $K$ of $K_\R$.)  Since
we have assumed $G_\R$ is equal rank, $\sigma = \Int(\xi)$ for
a semisimple element $\xi \in G$ whose square is the identity.  

For every $1\le i\le k$, define $\pi_k(\bar s_i) \in \End_\C(X\otimes V^{\otimes k})$
by
\begin{equation}
\label{e:barsi}
   \pi_k(\bar s_i)(x\otimes v_1\otimes \dots\otimes
v_i\otimes\dots\otimes v_k)=-x\otimes v_1\otimes\dots\otimes \xi
v_i\otimes\dots\otimes v_k.
\end{equation}
Clearly $\pi_k(\bar s_i)^2=\Id$, and it is easy to see that together with
$\pi_k(s_{i,j})$ of \eqref{e:Sact},
$\pi_k(\bar s_i)$ generate an action of $W(B_k)$, the Weyl group of type $B_k$ on $X\otimes
V^{\otimes k}.$ Since $K_\R$ commutes with $\xi$ we see that this
action factors through to $\wt F_{\mu,k,V}(X).$

We need to examine how $\pi_k(\bar s_i)$ interacts with the various $\pi_k(\Omega_{i,j})$.
To make computations, we choose a $\kappa$-dual basis $\B$ of $\frg$ such that
each individual $E \in \B$ is either in $\frk$ or $\frp$.  This is possible since we
have assumed $G_\R$ is equal rank: we choose a Cartan subalgebra $\frh$ of $\frk$,
which is also a Cartan subalgebra in $\frg$, 
and then choose the basis of $\frg$ to consist of either (suitably normalized)
elements of
$\frh$ or else (suitably normalized) root vectors. 
As before,
we adhere to Convention \ref{c:cas}.

We have the following calculation in $\End(X\otimes V^{\otimes k})$:
\begin{align*}
-\pi_k(\bar s_j)\pi_k(\Omega_{i,j})&= \pi_k\left(\sum_{E \in \B} (E)_i\otimes (\xi
E^*)_j\right ) \\
&= \pi_k\left(\sum_{E \in\B \cap \frk} (E)_i\otimes (E^*\xi)_j-\sum_{E \in \B \cap \frp}
(E)_i\otimes (E^*\xi)_j\right ).
\end{align*}
This implies
\begin{align}\label{eq:omK}
\pi_k(\bar s_j)\pi_k(\Omega_{i,j})+
\pi_k(\Omega_{i,j})\pi_k(\bar s_j)
=-2 \pi_k(\bar s_j)\pi_k(\Omega_{i,j}^{\frk}),\quad \text{where }\Omega_{i,j}^{\frk}=\sum_{E\in
  \B \cap \frk}(E)_i\otimes (E^*)_j.
\end{align}

\begin{lemma} Let $G_\R = U(p,q), Sp(2n,\R)$ or $O(p,q)$.
Fix a basis $\B$ for $\frg$ as described above, and let $C^{\frk}$ denote
the Casimir-type element for $\frk$, $\sum_{E \in B \cap \frk} EE^* \in U(\frk)$.
Fix a one-dimensional representation $\mu$ of $K_\R$ and
for  $x \in \frk$ write $\mu(x)$ for the complexification of the
differential of $\mu$.
Recall the  definition of
$\pi_k$ given in \eqref{e:Sact} and \eqref{e:barsi}; also write
$\pi_k(\ep_i)$ for the action of $\ep_i$ given in Proposition \ref{p:classHk}.
We have the following identity (as operators on $\wt
  F_{\mu,k,V}(X)$):
\begin{equation}\label{e:sbar}
\pi_k(\bar s_k)\pi_k(\ep_k)+\pi_k(\ep_k)\pi_k(\bar s_k)
=
2\pi_k(\bar s_k) \pi_k\left ( (C^{\frk})_k-\sum_{E\in B \cap \frk} (E^*)_k~ \mu(E) \right ).
\end{equation}
\end{lemma}

\begin{proof}
Recall $\pi_k(\ep_k)=\Omega_{0,k}+\Omega_{1,k}+\dots+\Omega_{k-1,k},$.
So we apply (\ref{eq:omK}) repeatedly to find:
\begin{align*}
\pi_k(\bar s_k)
\pi_k(\ep_k)
+\pi_k(\ep_k)\pi_k(\bar s_k)&=
-2\pi_k(\bar s_k)\pi_k\left (\sum_{i=0}^{k-1} \Omega_{i,k}^{\frk}\right ) \\
&=-2\pi_k(\bar s_k)\pi_k\left (\sum_{E\in B\cap \frk} (E^*)_k E -\sum_{E\in B \cap \frk} (E^*E)_k\right )\\
&=-2\pi_k(\bar s_k)\pi_k\left (\sum_{E\in B \cap \frk} (E^*)_k \mu(E)- (C^{\frk})_k \right).
\end{align*}
\end{proof}

It remains to compute the operator in the right hand side of
(\ref{e:sbar}).  Since $\mu$ is one-dimensional, $\mu(E)=0$ unless
$E\in\frz$, the center of $\frk$.   Equation (\ref{e:sbar}) becomes $\pi_k(\bar
s_k)\pi_k(\eps_k)+\pi_k(\eps_k)\pi_k(\bar s_k)=
2\pi_k(\bar s_k)\pi_k \left ((Q^{\frk}_\mu)_k\right),$ where 
\begin{equation*}
Q_\mu^{\frk}=C^{\frk}-\sum_{E\in B \cap \frz}\mu(E) E^*.
\end{equation*}

\begin{lemma}\label{l:param}Assume that $G_\R$ is one of the groups $U(p,q),$
  $Sp(2n,\R)$ or $O(p,q)$, $p\ge q$, and let $V$ be the defining representation of $G_\R$.
  Recall that conjugation by the semisimple element $\xi \in G$ defines the holomorphic automorphism
  of $G$ whose fixed points are the complexification of $K_\R$.
        Let $\mu$ be a character of $K_\R.$ Then there exist
  unique scalars $r_\mu$ and $c_\mu$ such that, as operators on $V$,
  \begin{equation*}
  Q_\mu^{\frk}-r_\mu=c_\mu\xi;
  \end{equation*} 
  that is, the action of the left-hand side on $V$ coincides with multiplication in $V$
  by $c_\mu \xi$.
The explicit cases are listed in Table \ref{t:cmu}.
\begin{table}[h]
\caption{List of characters $\mu$ and corresponding parameters. \label{t:cmu}}
\begin{tabular}{|c|c|c|c|c|}
\hline
$G_\R$ &$K_\R$  &$\mu$ &$r_\mu$ &$c_\mu$\\
\hline
$U(p,q)$ &$U(p)\times U(q)$ 
&$\det_{U(p)}^{m_p}\boxtimes\det_{U(q)}^{m_q}$ &$\frac{p+q-(m_p+m_q)}2$ &$\frac{p-q+m_q-m_p}2$\\
\hline
$Sp(2n,R)$ &$U(n)$  &$\det^m$ &$n$
 &$m$\\
\hline
$O(p,q)$ &$O(p)\times O(q)$ 
 &$\sgn_{O(p)}^{m_p}\boxtimes\sgn_{O(q)}^{m_q}$ &$\frac{p+q}2-1$
 &$\frac{p-q}2$\\
\hline
\end{tabular}
\end{table}
\end{lemma}

\vskip 0.5in

\begin{theorem}\label{t:class}  Assume that we are in the setting of Lemma
  \ref{l:param}. Set $k$ to be the real rank of $G_\R$ and recall the character
  $\mu_0$ of Proposition \ref{p:muexistence}. 
Then for each object 
$X \in \HC_{\mu\otimes \mu^*_0}(G_\R)$, 
there is a natural $\wt \H_k(c_\mu)$ action on 
$\wt F_{\mu,k,V}(X)$ defined on generators by
\begin{align*}
\epsilon_l &\mapsto \text{composition with } \pi_k\left (\Omega_{0,l}+\Omega_{1,l}+\dots+\Omega_{l-1,l} +r_\mu\right ) \\
s_{i,i+1} &\mapsto \text{composition with } \pi_k(s_{i.i+1}) \\
\bar s_k &\mapsto \text{composition with multiplication by $-\xi$ in the $k$th copy of $V$}
\end{align*}
(cf.~Proposition \ref{p:classHk}).
Hence we obtain an exact functor
\begin{align*}
F_{\mu,k,V}\; : \; \HC_{\mu\otimes \mu^*_0}(G_\R) &\longrightarrow \wt \H_k(c_\mu)\text{-mod} \\
X &\longrightarrow \Hom_{K_\R} ( \mu, X \otimes V^{\otimes k} ).
\end{align*}
\end{theorem}

\begin{proof} The claim follows from Proposition \ref{p:classHk} and
  Lemma \ref{l:param} combined with equation (\ref{e:sbar}).
\end{proof}

In particular, by taking $\mu \otimes \mu^*_0 = \1$, we obtain functors for constituents of spherical
principal
series.  Combined with Corollary \ref{c:typeA} for
the type A cases, we have the following corollary.

\begin{corollary}\label{c:spher} Assume $G_\R$ is one of the groups $GL(n,\R),$
  $U(p,q)$, $Sp(2n,\R),$ $O(p,q)$, $p\ge q,$ and let $V$ be the
 defining representation of $G_\R$. Let $k$ be the real rank
  of $G_\R,$ let $\mu_0$ be the character of Proposition \ref{p:muexistence},
 recall the category $\HC_\triv(G_\R)$ of subquotients of spherical minimal
 principal series of $G_\R$ (Definition \ref{d:HCmu}), and finally recall the Hecke algebra $\H(G_\R)$ from Definition
  \ref{d:gaha}.  
There exists an exact
  functor $F_\1:
  \HC_\triv(G_\R)\rightarrow \H(G_\R)\text{-mod}$, $F_\1(X)=\Hom_{K_\R}[\mu_0, X\otimes V^{\otimes
   k}]$, given by taking $\mu = \mu_0$ in Theorem \ref{t:class}.
\end{corollary}

\subsection{An alternative presentations of $\H(\Psi,c)$}
\label{s:drin}
 The results in this subsection will be needed in
Sections \ref{s:3} and \ref{s:4}.  In the setting of Lemma \ref{l:param},
let $m_\xi$ denote multiplication by $\xi$ in $V$.   As usual, given an
operator on $V$, let $(A)_i$ denote the operator on $V^{\otimes k}$ which
is the identity in all positions except the $i$th where it is $A$.

\begin{lemma}
Retain the setting of Theorem \ref{t:class}.
As operators on $V\otimes V,$ we have:
\begin{enumerate}
\item $\pi_2(\Omega^{\frk})=\frac 12(R_{12}+(m_\xi)_2 R_{12}(m_\xi)_2)$, if $G_\R=U(p,q)$;
\item $\pi_2(\Omega^{\frk})=\frac 12(R_{12}+(m_\xi)_2 R_{12}(m_\xi)_2)-\frac 12(\dim V)
  \pr_\1$, if $G_\R=Sp(2n,\R)$ or $O(p,q)$, where $\pr_\1$ denotes the projection onto the trivial
  representation isotypic component of $V\otimes V.$ 
\end{enumerate}
\end{lemma}

\begin{proof}
The proof is an analogous calculation with Casimir elements as in
Lemma \ref{omegaclass}. We skip the details.
\end{proof}

As in the proof of Proposition \ref{p:classHk}, we then find that,
under the assumptions of Theorem \ref{t:class}, we have the following
identity of operators on $F_{\mu,k,V}$:
\begin{equation}\label{e:Kact}
\pi_k(\Omega_{i,j}^{\frk})=-\frac 12 \pi_k(s_{i,j}+\bar s_j s_{i,j} \bar s_j);
\end{equation}
on the right-hand side, we have once again used $\pi_k$ to denote the
action of $\wt \H(c_\mu)$ given in Theorem \ref{t:class}.
Define
\begin{equation}\label{wteps}
\Omega_{i,j}^{\frp}=\sum_{E\in B \cap \frp} (E)_i\otimes
(E^*)_j.
\end{equation}
Clearly, we have $\Omega_{i,j}=\Omega_{i,j}^{\frk}+\Omega_{i,j}^\frp.$  We
compute the action of $\pi_k(\Omega_{i,j}^{\frp})$ on $F_{\mu,k,V}$ in terms of that
of 
$\eps_j$ and $s_\beta$ (for $\beta$ a positive restricted root) as follows:
\begin{equation}
\begin{aligned}
\pi_k(\Omega_{i,j}^{\frp})&=\pi_k\left(\Omega_{0,j}-\Omega_{0,j}^{\frk}\right)\\
&=
\pi_k(\eps_j) - r_\mu+\sum_{i<j}
\pi_k(s_{i,j})-\sum_{E\in B \cap \frk}\pi_k\left ((E)_0\otimes (E^*)_j \right ) \quad \text{by Theorem
\ref{t:class}}\\
&=\pi_k(\eps_j)+\pi_k\left ( (Q_\mu^{\frk})_j\right )
-r_\mu\\&+\sum_{1\le i\neq j\le
  k}\pi_k(\Omega_{i,j}^{\frk})+\sum_{i<j}\pi_k(s_{i,j}) \quad \text{by the same
  calculation as for (\ref{e:sbar})}\\
&=\pi_k(\eps_j)-c_\mu\pi_k(\bar s_j)+\frac 12\sum_{i<j} \pi_k(s_{i,j})\\
&-\frac
12\sum_{i>j}\pi_k(s_{j,i})-\frac 12\sum_{i\neq j} \bar \pi_k(s_js_{i,j}\bar
s_j) \quad \text{ by Lemma \ref{l:param} and (\ref{e:Kact})}.
\end{aligned} 
\end{equation}
We have proved:

\begin{lemma}\label{l:epstilde}In the setting of Theorem \ref{t:class}, the elements
\begin{equation}  
\label{e:epstilde}
\eps_j-\frac 12\sum_{\beta\in R^+} \co(\beta) \eps_j(\beta)s_\beta
\end{equation}
of the Hecke algebra $\wt \H_k(c_\mu)$ act by $\pi_k(\Omega_{0,j}^\frp)$ on $F_{\mu,k,V}(X).$
\end{lemma}

The elements of the lemma will be relevant for us in Section \ref{s:4}
in the context of Hermitian forms and the natural $*$-operation on
$\wt\H_k(c_\mu).$ They are also significant for a different reason which
we now explain.  
In the setting of Definition \ref{d:gaha}, for any $f \in V^*$, define
a new element $\wt f \in \H = \H(\Psi,\co)$ by
\begin{equation}
\label{e:drinlus2}
\wt f = f-\frac 12\sum_{\beta\in R^+} \co(\beta)
f(\beta) s_\beta
\end{equation}
where $R^+$ is the system of positive roots corresponding to the fixed
simple system $S \subset R$.
In $\H$, these elements satisfy
\begin{equation}\label{drinfeld}
\begin{aligned}
&s_\al\wt f=\wt{s_\al \cdot f},\\
&[\wt f,\wt f']=\left [\frac 12\sum_{\beta\in R^+} \co(\beta) f(\beta)s_\beta
\; ,\; \frac 12\sum_{\beta\in R^+} \co(\beta) f'(\beta) s_\beta \right ].
\end{aligned}
\end{equation}
In fact the elements $\wt f$ for $f \in V^*$ together with $\C[W]$
generate $\H$ and the relations \eqref{drinfeld} provide an alternative
presentation of $\H$.
This is sometimes called the
Drinfeld presentation
of $\H(\Psi, \co)$ (after its introduction in a more general setting 
\cite{drinfeld}).

\section{Images of spherical principal series}\label{s:3}

The following is a companion to Definition \ref{d:stdR}.  
\begin{defn}
\label{d:stdH}
Let $\H = \H(G_\R)$ be the Hecke algebra
attached to $G_\R$ by Definition \ref{d:HGR}. 
Let $\nu \in \fra^*$ be
dominant with respect to the fixed simple roots $\Pi_\circ \subset \Phi_\circ$, 
and write $\bbC_\nu$ for the
corresponding $S(\fra)$ modules.  Define
\[
X_\1(\nu) = \H \otimes_{S(\fra)} \bbC_\nu.
\]
Let $\ol X_\1(\nu)$ denote the unique irreducible subquotient of $X_\1(\nu)$ containing the
trivial representation of $W_\R$.  
(Alternatively $\ol X_\1(\nu)$ is characterized as the unique irreducible quotient of $X_\1(\nu)$.)
\end{defn}

The main result of this section is as follows, and the remainder of the section is devoted
to proving it.  (A more general result for $GL(n,\R)$ is proved 
by a different method in \cite[Theorem 3.5]{CT1}.)  

\begin{theorem}
\label{t:sphcomp}
Retain the setting of Corollary \ref{c:spher} and recall the standard modules of Definitions
\ref{d:stdR} and \ref{d:stdH}.  
Then
\[
F_\1(X^\R_\1(\nu)) = X_\1(\nu).
\]
\end{theorem}

\subsection{Preliminary details on $(\mu^*_0 \otimes V^{\otimes k})$ as a representation of $K_\R$.}
\label{s:singlepetaled}
We retain the notation from Section \ref{s:realform}. 
Fix root vectors $X_\al\in\frg_\al,$ for every $\al\in\Pi_\circ$.
  normalized such that
\begin{equation}
\kappa(X_\al,\theta(X_{\al}))=-2/||\al||^2,
\end{equation}
where $||\al||$ is the length of $\al$ induced by $\kappa.$
Set 
\begin{equation}\label{Zal}
Z_\al=X_\al+\theta(X_\al)\in \frk_\R.
\end{equation}
For later use define
\begin{equation}
\label{e:kal}
k_\al=\exp(\pi  Z_\al/2)\in K_\R.
\end{equation}
Then $k_\al^2\in M_\R,$ and $k_\al$ induces in $W_\R$ the reflection
corresponding to the root $\al.$

Let $(\tau, V_\tau)$ be any representation of $K_\R$. 
From the definitions (cf.~\eqref{WR}), $W_\R$ acts naturally on $V_\tau^{M_\R}.$
For every $Z_\al$, the operator $\tau(Z_\al)^2$ preserves $V_\tau^{M_\R}$ and it acts with
negative square integer eigenvalues on it.

\begin{definition}[Oda \cite{oda}]
\label{d:oda} 
A representation $(\tau, V_\tau)$ of $K_\R$ is called quasi-spherical if
$V_\tau^{M_\R} \neq 0$.
A quasi-spherical representation $(\tau,V_\tau)$ is
  called single-petaled if 
\begin{equation}
\label{e:odadef}
V_\tau^{M_\R}\subset \ker \left [ \tau(Z_\al)(\tau(Z_\al)^2+4) \right ].
\end{equation}
for all $\alpha \in \Pi_\circ$.  (A similar definition appears in
\cite{ba2} and \cite{ba3}, where the terminology ``petite'' is instead
used.  See also \cite{gu}.)
\end{definition}

\begin{proposition}\label{p:petite}Let $G_\R$ be one of the groups $GL(n,\R)$,
  $U(p,q)$, $Sp(2n,\R)$, $O(p,q)$, $p\ge q,$ and recall the
  $K_\R$-character $\mu_0$ from Proposition \ref{p:muexistence}. Let $V$ be the defining
  representation of $G_\R$, and let $k$ be the real rank
  of $G_\R.$  Then $(\mu^*_0\otimes V^{\otimes k})$ is a single-petaled representation
  of $K_\R$
  (Definition \ref{d:oda}). 
\end{proposition}

\begin{proof}[Sketch]
One verifies the claims in every case, by considering explicit
realizations for the elements $Z_\al$. We use the same realizations
and notation as in the proof of Lemma \ref{l:hom}, where we computed
the 
basis elements for $(\mu^*_0\otimes V^{\otimes k})^{M_\R}$. The
verification is reduced to the case of real rank one groups
corresponding to the restricted roots for these groups.

\smallskip

(a) $GL(2,\R).$ Here $\frp_\R$ consists of symmetric matrices, and we can choose
$\fra_\R$ to be the diagonal matrices. In the usual coordinates, the simple
restricted root is $\al=\eps_1-\eps_2,$ with corresponding
$X_\al=E_{12}.$ (Here and below $E_{i,j}$ denotes a matrix with a single nonzero
entry of $1$ in the $(i,j)$th position.)  We get $Z_{\al}=E_{12}-E_{21}$. We
compute the action of $Z_\alpha(Z_\alpha^2 +4)$ step-by-step as follows,
\begin{align*}
e_1\otimes e_2\xrightarrow{Z_{\al}}e_1\otimes e_1-e_2\otimes
e_2\xrightarrow{Z_\al}-2 e_1\otimes e_2 -2 e_2\otimes e_1\xrightarrow{+4}2(e_1\otimes e_2-e_2\otimes e_1)\xrightarrow{Z_\al}0,
\end{align*} 
and find $Z_\alpha(Z_\alpha^2 +4)$ indeed acts by zero on the basis vector $e_1 \otimes e_2$
appearing in the proof of Lemma \ref{l:hom}(a).  The identical considerations give the same conclusion
for $e_2 \otimes e_1$, as desired.

(b) $GL(2,\C).$ The notation and calculation are identical to the case
of $GL(2,\R).$ 

(c) $U(p,1), p\geq 1.$  In this case,
$K_\R$ consists of the block-diagonal $U(p)\times U(1)$, and $\frp_\R$ is
formed of matrices
$\left ( \begin{matrix}0&B \\
  B^*&0\end{matrix}\right )$
where $B$ is a $p \times 1$ complex matrix and $B^*$ denotes its conjugate
transpose.  We take $\fra_\R \subset \frp_\R$ to consist of matrices
where $B$ is of the form $(0,0,\dots, x)^{tr}$ for $x \in \C$.  Let 
$\eps \in \fra_\R^*$ denote the functional which takes value $x$ on such 
an element of $\fra_\R$.  Then the roots of $\fra_\R$ in $\frg_\R$
are simply $\pm \eps$ and $\pm 2\eps$, and so we take $\Pi_\circ = \{\eps\}$.
Then $Z_\eps = 2(E_{1,p} + E_{p,1})$.
We compute the action of $Z_\eps(Z_\eps^2  + 4)$ on $f_1^\pm = e_{p} \pm e_{p+1}$
sequentially as
\begin{align*}
f_1^\pm\xrightarrow{Z_{\eps}} 2e_1 \xrightarrow{Z_{\eps}}-4e_{p}\xrightarrow{+4}
-4e_p + 4f_1^\pm = \pm 4e_{p+1}\xrightarrow{Z_{\eps}} 0.
\end{align*}

(d) $O(p,1), p>1.$ We have realized $O(p,1)$ naturally as a subgroup
of $U(p,1),$ and so the calculation is essentially the same as in (c).
\end{proof}

In the setting of Proposition \ref{p:petite}, we next note that
$(\mu^*_0 \otimes V^{\otimes k})^{M_\R}$
has two natural actions of $W_\R$ on it.  One comes from the definitions and is computed
in Proposition \ref{p:muexistence}.  The other comes from identifying the space
$(\mu^*_0 \otimes V^{\otimes k})^{M_\R}$ with $F_\1(X^\R_\1(\nu))$ using \eqref{e:frob}, and then
restricting the action of $\H(G_\R)$ given in Corollary \ref{c:spher} to the subalgebra
$\bbC[W_\R]$. 

\begin{prop}
\label{p:wact}
The two natural actions of $W_\R$ on $(\mu^*_0 \otimes V^{\otimes k})$ described
in the previous paragraph coincide.
\end{prop}

\begin{proof}[Sketch]
On one hand, the action of $\C[W_\R] \subset \H(G_\R)$ is given
in Theorem \ref{t:class} and can be easily computed
on the explicit bases constructed in the proof of Lemma \ref{l:hom}.

On the other hand, the elements $k_\alpha \in K_\R$ defined in \eqref{e:kal}
give the other action of $W_\R$.  To prove the proposition, we need to 
compute the $k_\alpha$ explicitly and check that their action on the bases
in the proof of Lemma \ref{l:hom} coincides with that in the previous 
paragraph.  This is an explicit case-by-case calculation (reducing again
to real rank one) and we simply
sketch the details in certain representative cases.  
With notation (especially for the elements $Z_\al$)
as in the proof of Proposition \ref{p:petite}, we have:

(a) $GL(2,\R)$.  We have $k_{\al}=\exp(\frac{\pi}2 Z_{\al})=
E_{12}-E_{21} \in O(2)$.  It takes $e_1 \otimes e_2$ to
$-e_2 \otimes e_1$, and 
so $s_\al$ acts by $-R_{12}$ just as in the restriction of the
Hecke
algebra action detailed in Corollary \ref{c:typeA} (cf.~\eqref{e:Sact}).

(c) $U(p,1)$.
In this case, $k_{\eps}$ is the diagonal matrix
$(-1,1,\dots,1,-1,1)$ with negative entries in the first
and $p$th entries.  Thus $k_{\eps_1}$ takes 
$f_1^\pm = e_p \pm e_{p+1}$ to $-f_1^\mp$.  This coincides with the
action of multiplication by $-\xi$ as in the restriction of the Hecke
algebra action described in the third displayed equation of Theorem 
\ref{t:class}. 

\end{proof}

\subsection{A natural vector space isomorphism $F_\1(X^\R_\1(\nu)) \rightarrow X_\1(\nu)$}
\label{s:vs}
From Lemma \ref{l:hom}(2), we know that $F_\1(X^\R_\1(\nu))$ and $X_\1(\nu)$
have the same dimension (namely that of $|W_\R|$).  In this section, we find 
a natural map giving the isomorphism of vector spaces.  In the next section,
we finish the proof of Theorem \ref{t:sphcomp} by showing that this natural
map is an $\H$-module map.

We begin by working in the setting of general $G_\R$.
Consider the ``universal'' spherical $(\frg,K)$-module
\begin{equation}
\label{e:XR}
\caX^\R=U(\frg)\otimes_{U(\frk)} \triv.
\end{equation}
In particular, there is an obvious map $\caX^\R \to X_\1^\R(\nu)$
for any $X_\1^\R(\nu)$ as in Definition \ref{d:stdR}.  By the dominance
condition on $\nu$, the map is surjective.  Write $\caJ^\R(\nu)$ for
its kernel.  So $X_\1^\R(\nu) \simeq \caX^\R/\caJ^\R(\nu)$.

Set $\H = \H(G_\R)$ (Definition \ref{d:HGR}) and consider the $\H$-module
analog
\begin{equation}
\label{e:XH}
\caX=\H\otimes_{C[W_\R]} \triv.
\end{equation}
We will now work exclusively with the Drinfeld presentation of $\H$
given around \eqref{drinfeld}.  In the present case $V^* = (\fra^*)^* = \fra$.
Given $h =x_1x_2\cdots x_l \in S(\fra)$ with each $x_i \in \fra$, 
we define
\begin{equation}
\label{e:XHaction1}
\wt h \otimes \1 = (\wt x_1 \wt x_2 \cdots \wt x_l) \otimes 1 \in \caX.
\end{equation}
Because the commutator of any $\wt x_i$ and $\wt x_j$ is contained
in $\C[W_\R]$ by \eqref{drinfeld}, the notation is well-defined
(independent of the ordering of the $x_i$).  Thus, as
sets, $\caX = \{\wt h \otimes \1 \;
| \; h \in S(\fra) \}$.
With this identification, for $f \in \fra$ and correspondingly
$\wt f \in \H$, the action of
$\wt f$ on $\caX$ is given 
\begin{equation}
\label{e:XHaction}
\wt f \cdot (\wt h \otimes 1) = (\wt{fh} \otimes 1).
\end{equation}
Again there is an obvious surjection $\caX \to X_\1(\nu)$
for any $X_\1(\nu)$ as in Definition \ref{d:stdH}.  
Write $\caJ^\R(\nu)$ for its kernel, so that
 $X_\1(\nu) \simeq \caX/\caJ(\nu)$.
 
 We now introduce
 \begin{equation}
 \label{e:hc}
\gamma \; : \;  \caX^\R {\lra} \caX.
 \end{equation}
Let
 $ \gamma_\circ \; : \; U(\frg) \lra U(\fra) \simeq S(\fra)$
 denote the projection with respect to the decomposition
 $U(\frg) = U(\fra) \oplus (\frn U(\frg) + U(\frg)\frn)$.
Define
 \begin{equation}
 \label{e:hc2}
\gamma (x \otimes 1)
 = \wt{\gamma_\circ(x)} \otimes 1;
\end{equation}
here we are again using the identification $\caX = \{\wt h \otimes \1 \;
| \; h \in S(\fra) \}$ .  (The elements $\wt h$
have a $\rho$-shift built into them, cf.~\eqref{e:drinlus2}, so $\gamma$ is
really a kind of Harish-Chandra homomorphism.)

 Suppose now that $(\tau,V_\tau)$ is any representation
 of $K_\R$.  
 Consider the map
 \begin{equation}
 \label{e:oda}
 \Gamma\; : \; \Hom_{K_\R}(V_\tau,  \caX^\R) \lra \Hom_{W_\R}(V_\tau^{M_\R}, \caX)
 \end{equation}
 defined as the composition of restriction to $V_\tau^{M_\R}$ and 
composition with 
 $\gamma$ of \eqref{e:hc} and \eqref{e:hc2}.   Here is the connection with the previous section.
 
 \begin{theorem}[\cite{oda}] 
 \label{t:oda}
Let $(\tau,V_\tau)$ be a single-petaled representation of $K_\R$ (Definition
\ref{d:oda}). Then the map $\Gamma$ of \eqref{e:oda} is an isomorphism.
Moreover it factors to an injection
 \begin{equation}
 \label{e:odanu}
 \Gamma_\nu \; : \; \Hom_{K_\R}(V_\tau,  \caX^\R/\caJ^\R(\nu)) 
 \hookrightarrow \Hom_{M_\R}(V_\tau^{M_\R}, \caX/\caJ(\nu))
 \end{equation}
for each $\nu \in \fra^*$.
\end{theorem}

\begin{proof}
The first assertion is \cite[Theorem 1.4]{oda}.  The second follows from
\cite[Remark 1.5]{oda}.
\end{proof}
 
 \medskip
 
 Return to the setting of Theorem \ref{t:sphcomp}.  Take $V_\tau
 = (\mu^*_0 \otimes V^{\otimes k})^*$.  This is single-petaled by 
 Proposition \ref{p:petite}, and so Theorem \ref{t:oda} applies
 to give an inclusion
 \[
\Hom_{K_\R}\left ((\mu^*_0 \otimes V^{\otimes k})^*
,  \caX^\R/\caJ^\R(\nu)\right ) 
 \hookrightarrow \Hom_{W_\R}
\left([(\mu^*_0 \otimes V^{\otimes k})^*]^{M_\R}, \caX/\caJ(\nu)\right).
\]
The left-hand side clearly identifies with $F_\1(X_\1^\R(\nu))$.  Because
of Proposition \ref{p:muexistence}, the right-hand side identifies
with $X_\1(\nu)$.   The dimension count of Lemma \ref{l:hom}(2) implies that the injection
is an isomorphism.  We thus obtain a vector space isomorphism
\begin{equation}
\label{e:gammanu}
\Gamma_\nu \: : \; 
\Hom_{K_\R}\left ((\mu^*_0 \otimes V^{\otimes k})^*
,  \caX^\R/\caJ^\R(\nu)\right ) 
 \stackrel{\sim}{\lra} \Hom_{W_\R}
\left( \C[W_\R], \caX/\caJ(\nu)\right).
\end{equation}

\subsection{$\Gamma_\nu$ is an $\H$-module map.}
Corollary \ref{c:spher} gives a natural action of $\H$ on the
left-hand side of \eqref{e:gammanu}.  Because the right-hand side
identifies with $X_\1(\nu)$, it also has a natural $\H$ action.
In the this section we make those actions explicit, and then
finish the proof of Theorem \ref{t:sphcomp} by
\eqref{e:gammanu} respects the $\H$-action.

To begin, we recall Proposition \ref{p:muexistence} and
fix a vector $\bfv \in (\mu^*_0 \otimes V^{\otimes k})^*\simeq \C[W_\R]$ 
which is cyclic for the action of $W_\R$.  Though not essential, we save
ourselves some notation by noting that in each of the classical cases
under consideration, it is not difficult to verify using the bases introduced
in the proof of Lemma \ref{l:hom}(2) that $\bfv$ may be taken to be a
simple tensor,
\[
\bfv = \lam \otimes \lam_1 \otimes \cdots \lam_k;
\]
here $\lam \in \bbC_{\mu^*_0}^*$ and $\lam_i \in V^*$.
For later use, we fix a basis $\{E_1, \dots, E_k\}$ of $\fra$
such that $E_i\lam_j = \delta_{ij}$, the Kronecker delta.
Fix $\Upsilon$ in the left hand side of \eqref{e:gammanu}.
We need to explicitly understand the action of $\H$ on $\Upsilon$.
Again we work with the presentation of $\H$ given in \eqref{drinfeld}.
In particular, we have the elements $\wt f_i := \wt E_i$
whose span over $\C$ in $\H$ coincides with the span of 
all the elements of the form $\wt f$, $f \in \fra$.
Then Lemma \ref{l:epstilde} (and the discussion at the end
of Section \ref{s:drin}) show that $\wt f_i$ 
acts in the $\H$ module $F_\1(X)$ by operator
$\Omega_{0,i}^\frp$ defined in \eqref{wteps}.  Unwinding the natural vector
space isomorphism between $F_1(X^\R_\1(\nu))$ and the left-hand
side of \eqref{e:gammanu}, the value at $\bfv$
of the function obtained by action by $\wt f_i$ on $\Upsilon$ is
\begin{equation}
\label{e:lefthandaction}
[\wt f_i \cdot \Upsilon](\bfv) = 
\sum_{E \in \B \cap \frp} E\Upsilon(\lam \otimes \lam_1 \otimes \cdots \otimes 
E^*\lam_i \otimes \cdots \otimes \lam_k),
\end{equation}
where the actions of $E \in \frp$ on the right-hand side are the
obvious ones.  Recall, as in Section \ref{s:heckeB}, that the
elements $E \in \B \cap \frp$ are a basis of $\frp$ which are (orthonormal)
simultaneous eigenvectors for the action of $\frh = \frt \oplus \fra$.
Beyond that requirement and normalization considerations, 
we are free to choose them as we wish,
and so we may assume that (possibly after rescaling)
$E_1, \dots, E_k$ are among them.

Because of the appearance of the projection $\gamma_\circ$
in the definition of $\Gamma_\nu$, many terms in \eqref{e:lefthandaction}
will not contribute to  $\Gamma_\nu(\wt f_i \cdot \Upsilon)$.  
More precisely, note that for any $\Psi$ in the left-hand
side of \eqref{e:gammanu}, $\Gamma_\nu(\Psi)$ is determined by
its value at $\bfv$; moreover, this value is 
$[\gamma_\circ(\Psi(\bfv))]\wt{\phantom{x}} 
\otimes 1$.   Because the terms involving $E\notin \fra$ will not
survive $\gamma_\circ$, we see
\[
\left [ \Gamma_\nu(\wt f_i \cdot \Upsilon)\right ](\bfv)
=
\left [\gamma_\circ  \left (
\sum_{E \in \B \cap \fra} E\Upsilon(\lam \otimes \lam_1 \otimes \cdots \otimes 
E^*\lam_i \otimes \cdots \otimes \lam_k)
\right)\right ] ^{\wt{\phantom{x}}}\; \otimes \; 1.
\]
Thus we have
\begin{align*}
\label{e:action}
\left [ \Gamma_\nu(\wt f_i \cdot \Upsilon)\right ](\bfv)
&=
[\gamma_\circ (E_i\Upsilon(\bfv))]^{\wt{\phantom{x}}} \otimes 1
 = [E_i \gamma_\circ(\Upsilon(\bfv))]^{\wt{\phantom{x}}} \otimes 1 \\
&= \wt f_i \cdot [\Gamma_\nu(\Upsilon)\bfv];
\end{align*}
the last equality is by \eqref{e:XHaction}.
Thus \[
\Gamma_\nu(\wt f_i \cdot \Upsilon) = \wt f_i \cdot \Gamma_\nu(\Upsilon).
\]
and $\Gamma_\nu$ respects the action of the elements 
$\wt f_i$, hence all $\wt f$, in $\H$.

 It remains to check that $\Gamma_\nu$ is equivariant for the action of $W_\R$.
 But this reduces to Proposition \ref{p:wact}.   We omit further
details.  This completes the proof of Theorem
 \ref{t:sphcomp}
 \qed

\section{Hermitian forms}\label{s:4}

This
section can be interpreted as a generalization of results of
\cite{suzuki} for category $\caO$ for $\frg\frl(n)$.  The main result is Theorem 
\ref{t:herm}.

\subsection{Generalities on Hermitian forms}

\begin{definition} [e.g.~\cite{v:unit}] 
\label{d:invR}
In the setting of a general
real reductive group $G_\R$, let $X$ be a $(\g,K)$-module. A Hermitian form
$\langle~,~\rangle:X\times X\to\C$ is called invariant if it satisfies:
\begin{enumerate}
\item[(a)] $\langle A\cdot x,y\rangle=-\langle x, \overline A\cdot
  y\rangle,$ for every $x,y\in X,$ $A\in\g$; here, $\overline A$
  denotes the complex conjugate of $A$ (with respect to $\g_\R$).
\item[(b)] $\langle k\cdot x,y\rangle=\langle x, k^{-1}\cdot
  y\rangle,$ for every $x,y\in X,$ $k\in K.$
\end{enumerate}
A $(\frg,K)$ module $X$ is called Hermitian if there exists an invariant 
form on $X$.  (In this case, the form is unique up to scalar.)  If the
form is positive definite, then $X$ is called unitary.
\end{definition}

\begin{definition} [\cite{BM2}]
\label{d:invH}
In the setting of the Definition \ref{d:gaha}, set $\H = \H(\Psi,\co)$ and
let $V$ be a $\H$-module. A Hermitian form on $V$ is $\H$-invariant if 
\begin{equation}
\langle x\cdot u,v\rangle=\langle u,x^*\cdot v\rangle, \text{ for all
 } u,v\in V,\ x\in\H, 
\end{equation}
where $*$ is an involutive anti-automorphism on $\H$ defined on
generators in the Lusztig presentation by:
\begin{align*}
w^*&=w^{-1},\text{ for all }w\in W,\\
f^*&=-f+\sum_{\beta\in R^+}c_\beta \langle
f,\beta\rangle s_\beta,\text{ for all }f\in \fra,
\end{align*}
where the sum ranges over the positive roots, and $s_\beta$ denotes
the reflection with respect to $\beta.$  Because of 
\eqref{e:drinlus2}, if we instead use the Drinfeld presentation
\eqref{drinfeld} of $\H$, $*$ is defined by
\begin{align*}
w^*&=w^{-1},\text{ for all }w\in W,\\
\wt f^*&=- \wt f \; \text{ for all }f\in \fra.
\end{align*}
\end{definition}

\begin{example}
\label{e:star}
In the case of the type $C_k$ Hecke algebra $\wt \H_k(c)$ appearing
in Example \ref{e:tHk}, we get:
\begin{equation}\label{starB}
\begin{aligned}
&s_{i,j}^*=s_{i,j},& 1\le i<j\le k;\\
&\bar s_j= \bar s_j, & 1\le j\le k;\\
&\wt \epsilon_j^*=-\wt \epsilon_j, & 1\le j\le k.
\end{aligned}
\end{equation}
\end{example}

\subsection{Preservation of unitarity}
Return to the setting of Corollary \ref{c:spher}.
Given a $(\g,K)$-invariant Hermitian form $\langle \; , \; \rangle_X$
on an object $X$ in $\HC_\1(G_\R)$, we wish to construct
a $\H(G_\R)$-invariant form on $F_\1(X)=\Hom_{K_\R}[\mu_0,X\otimes V^{\otimes
  k}].$ To get started, we fix a positive definite Hermitian inner product 
$(~,~)_V$ on $V$ such that:
\begin{align}
&(k\cdot u,v)_V=(u,k^{-1}\cdot v)_V,\text{ for all } k\in K_\R,\\\notag
&(E\cdot u,v)_V=(u,E\cdot v)_V,\text{ for all } E\in\frp_\R.
\end{align}
Define a Hermitian form on $\Hom_\C(\mu_0,X\otimes V^{\otimes k})$ as follows.
Given two elements $\varphi, \psi$ with images
$\varphi(1) = x\otimes u_1\otimes\dots\otimes u_n$ and $\psi(1)
= y\otimes v_1\otimes\dots\otimes v_n.$, set
\begin{equation}\label{form}
\langle \varphi, \psi \rangle = \langle x,y\rangle_X (u_1,v_1)_V\dots (u_n,v_n)_V.
\end{equation}
Extend linearly to arbitrary $\varphi$ and $\psi$.
For every $a\in K_\R$ and $A\in \frk_\R$, we have:
\begin{align}
\label{e:invform}
\langle a\cdot \varphi, \psi \rangle = \langle \varphi, a^{-1} \cdot \psi \rangle \\
\langle A\cdot \varphi, \psi \rangle = -\langle \varphi, A \cdot \psi \rangle
\end{align}
Thus the form (\ref{form}) induces a Hermitian
form,  denoted $\langle~,~\rangle_\1$, on the trivial $K_\R$ isotypic component $F_\1(X)$
of $\Hom_\C(\mu_0,X\otimes V^{\otimes k})$. 
It remains to check that this is $\H(G_\R)$-invariant.

\begin{lemma}\label{l:herm} 
In the setting of Corollary \ref{c:spher}, assume
  that $F_\1(X)\neq 0.$  Assume further that there exists
  an invariant form $\langle~,~\rangle_X$ on $X$ (Definition
  \ref{d:invH}). 
  Then the Hermitian form $\langle~,~\rangle_\1$ defined
  on $F_\1(X)$
  induced by (\ref{form}) is $\H(G_\R)$-invariant. Moreover, if
 $\langle~,\rangle_X$ is nondegenerate, then $\langle~,~\rangle_\1$
 is nondegenerate.
 \end{lemma}

\begin{proof} We verify that $\langle~,~\rangle_\1$
  preserves the $*$-operation on the generators of $\H(G_\R)$
  in the Drinfeld presentation as in Example \ref{e:star}. Since
  $s_{i,j}$ acts by permuting the factor of $V^{\otimes k}$-factors, we clearly have
$$\langle s_{i,j} \cdot \varphi, \psi \rangle=\langle \varphi, s_{i,j} \cdot \psi\rangle.$$
Also, since $\xi\in K_\R$ and $\xi^2=1,$ we have
$$\langle \bar s_{i} \cdot \varphi, \psi\rangle=\langle \varphi, \bar s_{i}\cdot \psi\rangle.$$
It remains to check the claim for each $\wt\eps_i$. Note that for every $1\le i\le k,$
\begin{align*}
\langle\pi_k(\Omega_{0,i}^{\frp})(x\otimes\dots u_i\dots),&(y\otimes\dots
v_i\dots)\rangle=\sum_{E\in\frp_\R}\langle Ex\otimes\dots
E^*u_i\dots, y\otimes\dots v_i\dots\rangle\\
&=\sum_{E\in\frp_\R}\langle
Ex,y\rangle\dots(E^*u_i,v_i)_V\dots=-\sum_{E\in\frp_\R}\langle x,Ey\rangle\dots(u_i,E^* v_i)_V\dots\\
&=-\langle(x\otimes\dots u_i\dots),\pi_k(\Omega_{0,i}^{\frp})(y\otimes\dots
v_i\dots)\rangle.
\end{align*} 
Since $\wt\eps_i$ acts by $\Omega_{0,i}^\frp$ (Lemma \ref{l:epstilde} and \eqref{e:drinlus2}), this
implies that
$$\langle\wt \eps_i \cdot \varphi, \psi \rangle=-\langle \varphi  ,\wt\eps_i \cdot \psi \rangle.$$
Hence the form is invariant.

Now suppose that $\langle~,~\rangle_X$ is nondegenerate, and suppose
the induced form $\langle~,~\rangle_\1$ on $F_\1(X)$ were degenerate. 
Notice that \eqref{e:invform} shows that the trivial isotypic component
$F_\1(X)$ is orthogonal with (respect to $\langle~,~\rangle_X$) 
to every nontrivial $K_\R$ isotypic component of $\Hom_\C(\mu_0, X \otimes V^{\otimes k})$.
Thus if $F_\1(X)$ were degenerate, then $\langle~,~\rangle_X$ would also be degenerate,
a contradiction.
\end{proof}

\begin{theorem}\label{t:herm} In the setting of Corollary \ref{c:spher} and Theorem
\ref{t:sphcomp},
let  $X$ be an
  irreducible Hermitian spherical $(\frg,K)$-module.  Then:
\begin{enumerate}
\item $F_\1(X)$ is an irreducible Hermitian spherical $\H(G_\R)$-module, and
\item if, in addition, $X$ is unitary, then $F_\1(X)$ is unitary.
\end{enumerate}
\end{theorem}

\begin{proof}
In light of Lemma \ref{l:herm}, the only claim that needs explanation
is the preservation of irreducibility. So let $X$ be an irreducible
Hermitian 
spherical 
$(\frg,K)$-module in $\caH_\1(G_\R)$.  For the groups under consideration,
$X$ is of the form $\ol X^\R_\1(\nu)$ with notation as in Definition \ref{d:stdR};
see \cite{He}, for instance.
More precisely, there exists an invariant Hermitian form, say $\langle ~,~\rangle_X$, on $X_\1^\R(\nu)$
such that $\ol X_\1^\R(\nu)$ is the quotient of $X_\1^\R(\nu)$ by the radical
of $\langle ~,~\rangle_X$; that is, there is an exact sequence
\[
0 \lra \mathrm{rad}\langle ~,~\rangle_X \lra X_\1^\R(\nu)  \lra \ol X_\1^\R(\nu) \lra 0,
\]
where $\mathrm{rad}\langle ~,~\rangle_X$ denotes the radical of $\langle ~,~\rangle_X$.
In particular, the form $\langle ~,~\rangle_X$ is nondegenerate on $\ol X_\1^\R(\nu)$.
Applying the exact functor $F_\1$ we have
\[
0 \lra F_1(\mathrm{rad}\langle ~,~\rangle_X) \lra X_\1(\nu)  \lra F_1(\ol X_\1^\R(\nu)) \lra 0
\]
where we have used Theorem \ref{t:sphcomp} on the middle term.
Lemma \ref{l:herm} gives an invariant form $\langle ~,~\rangle_\1$
on $F_\1(X^\R_\1(\nu)) = X_\1(\nu)$ and a corresponding
nondegenerate form on  $F_1(\ol X_\1^\R(\nu))$.  Thus the second
exact sequence is really
\begin{equation}
\label{e:exact}
0 \lra \mathrm{rad}\langle ~,~\rangle_\1 \lra X_\1(\nu)  \lra F_1(\ol X_\1^\R(\nu)) \lra 0
\end{equation}
It is easy to check that a nonzero 
spherical vector in $X^\R_\1(\nu)$ naturally
gives rise to a nonzero spherical vector in $F_\1(X_\1^\R(\nu))$.
Thus
$F_\1(\overline X^\R_\1(\nu))$ is nonzero.  Together with \eqref{e:exact},
this implies $F_1(\ol X_\1^\R(\nu))$ is irreducible.
\end{proof}

As remarked in Definition \ref{d:stdH}, the standard module $X_\1(\nu)$
(for $\nu$ dominant) has a unique irreducible quotient.  So 
\eqref{e:exact} gives the following more precise result.

\begin{corollary}
\label{c:hermcomp}
Retain the setting of Corollary \ref{c:spher} and recall the standard modules of Definitions
\ref{d:stdR} and \ref{d:stdH}.   Assume $\ol X^\R_\1(\nu)$ is Hermitian (Definition \ref{d:invR}).
Then
\[
F_\1(\ol X^\R_\1(\nu)) = \ol X_\1(\nu).
\]
\end{corollary}

\begin{remark}
In fact, Corollary \ref{c:hermcomp} remains true without the
assumption that $\ol X^\R_\1(\nu)$ is Hermitian.  (One instead 
needs to consider the pairing of $\ol X^\R_\1(\nu)$ with
its Hermitian dual and step through the proof of Theorem \ref{t:herm}.)
  Since the argument goes through
without much change, we omit the details.
\end{remark}

\end{document}